 \RequirePackage{fix-cm}

 \documentclass[smallextended]{svjour3} %
 \smartqed  %

\usepackage{amssymb,amsmath,amsfonts,latexsym} %
\usepackage{hyperref}

\usepackage{mathtools}
\usepackage{multirow,array}
\usepackage{graphicx}
\usepackage{subfigure}
\usepackage{epstopdf}
\usepackage{placeins}
\usepackage{float} %
\usepackage{color, xcolor}
\usepackage{nicefrac,comment}
\usepackage{siunitx}
\usepackage{pgfplots}
\usepackage{tikz}
\usetikzlibrary{plotmarks}
\pgfplotsset{compat=1.18} 
\usepackage{booktabs}
\usepackage{algorithm}
\usepackage{algpseudocode}

\newcommand{\mG}{\mathsf{G}} 
\newcommand{\G}{\mathcal{G}}
\allowdisplaybreaks

\newtheorem{algo}{Algorithm}

\journalname{Fract. Calc. Appl. Anal.} %

\begin{document}

\title{A \(\star\)-Product Approach for Analytical and Numerical Solutions of Nonautonomous Linear Fractional Differential Equations}

\titlerunning{A \(\star\)-Product Approach for Nonautonomous Linear FDEs}
\author{
        Fabio Durastante$^1$ %
\and
        Pierre-Louis Giscard$^2$ %
\and
        Stefano Pozza$^3$ %
 }

\authorrunning{F. Durastante \and P.L. Giscard \and S. Pozza} %

\institute{Fabio Durastante$^{1}$
\at
Dipartimento di Matematica, Università di Pisa, Largo Bruno Pontecorvo, 5 56127 Pisa (PI), Italy.\\
\email{fabio.durastante@unipi.it}\\
ORCID ID: 0000-0002-1412-8289.
 \and
Pierre-Louis Giscard$^{2}$
\at
Laboratoire de Mathématiques Pures et Appliquées Joseph Liouville (LMPA)
Université Littoral Côte d'Opale, Calais, France. \\
\email{giscard@univ-littoral.fr}\\
ORCID ID: 0000-0003-3025-8750
\and
Stefano Pozza$^{3,*}$
\at
Department of Numerical Mathematics, Faculty of Mathematics and Physics, Charles University, Sokolovská 83, 186 75 Praha 8, Czech Republic.\\
\email{pozza@karlin.mff.cuni.cz}$^*$ corresponding author\\ %
ORCID ID: 0000-0003-1529-8420
}

\date{Received: \today / Revised: \ldots / Accepted: \ldots}

\maketitle

\begin{abstract}
This article presents a novel solution method for nonautonomous linear ordinary fractional differential equations. The approach is based on reformulating the analytical solution using the $\star$-product, a generalization of the Volterra convolution, followed by an appropriate discretization of the resulting expression. Additionally, we demonstrate that, in certain cases, the $\star$-formalism enables the derivation of closed-form solutions, further highlighting the utility of this framework.


\end{abstract} %

\section{Introduction} \label{sec:1}

\setcounter{section}{1} \setcounter{equation}{0} %

We consider systems of Fractional Ordinary Differential Equations (FODE) of the form  
\begin{equation}\label{eq:fode}
    y^{(\alpha)}(t) = \widetilde{f}(t)y(t), \quad y(0) = y_0, \quad t \in I = [0,T] \subseteq \mathbb{R}, \quad 0 < \alpha \leq 1,
\end{equation}
where {$\tilde{f}(t)$ is a regular function of $t$,} $y^{(\alpha)}(t)$ denotes the Caputo fractional derivative of order $\alpha$~\cite{CaputoDerivative}, defined~as  
\[
{y^{(\alpha)}(t) = \frac{1}{\Gamma(n-\alpha)} \int_{0}^{t} (t-\tau)^{n-\alpha-1} y^{(n)}(\tau)\,\mathrm{d}\tau, \quad n-1 < \alpha \leq n, \; n \in \mathbb{N}}
\]
with $\Gamma(\cdot)$ representing the Euler Gamma function~\cite[\S 5]{NIST:DLMF} , given by its integral definition  
\[
\Gamma(z) = \int_0^\infty t^{z-1} e^{-t} \, dt, \quad \text{for } \Re(z) > 0,
\]
which extends meromorphically to the entire complex plane except at the negative integers.

Equations of this type arise in a wide range of applications, including the modeling of viscoelastic materials~\cite{BagleyTorvik}, suspension flows via hyperbolic models~\cite{MR2386503}, and problems in Earth system dynamics~\cite{MR3671993}. They also appear in reformulations of quantum mechanics using fractional calculus. In particular, the fractional Schr\"{o}dinger equation was introduced in~\cite{PhysRevE.66.056108} and further analyzed in~\cite{PhysRevE.62.3135}. An extension of the Feynman path integral that incorporates L\'{e}vy flight paths was developed in~\cite{MR1755089}, while a time-fractional formulation that generalizes conventional quantum mechanics was proposed in~\cite{MR3671992}. The unification of fractional mechanics with quantization procedures was explored in~\cite{MR2944107}, and the time-fractional Schr\"{o}dinger equation—where the standard first-order time derivative is replaced by a Caputo fractional derivative—was studied in~\cite{10.1063/1.1769611}, leading to solutions with intrinsic nonlocal temporal behavior. Complementing these theoretical developments, numerical methods for the fractional Schr\"{o}dinger equation were proposed in~\cite{MR3033699,MR3689930}, highlighting both the mathematical depth and practical relevance of these fractional formulations.

The usual starting point to derive a method for solving~\eqref{eq:fode} is through the reformulation of it into its 
equivalent nonlinear Volterra integral equation form
\begin{equation}\label{eq:volterra-solution}
\begin{split}
y(t) = &\; y_0 + \frac{1}{\Gamma(\alpha)} \int_{0}^{t} (t-\tau)^{\alpha-1} \widetilde{f}(\tau)y(\tau)\,\mathrm{d}\tau \\
= &\; y_0 + \prescript{\mathrm{RL}}{0}{I}_{t}^\alpha \widetilde{f}(t)y(t),
\end{split}
\end{equation}
where $\prescript{\mathrm{RL}}{a}{I}_{b}^\alpha \cdot $ is the Riemann--Liouville fractional integral on the $[a,b]$ interval; see, e.g., \cite{DiethelmBook,garrappa2018numerical}. Most of these methods are, in their original formulation, of low order which can typically be improved through extrapolation techniques, see, e.g., the discussion in~\cite{DIETHELM2006482}. In addition, the most used methods are of the finite difference type and return the solution as a set of evaluations on a given grid, consider, e.g., fractional predictor-corrector methods~\cite{MR1926466} and fractional linear multistep methods~\cite{MR3327641,garrappa2018numerical,MR804935}. All these methods can be framed in the formula ``\emph{first discretize and then solve}''. In fact, these generate a succession of approximate problems whose solution is a convergent succession to the solution. The other possible approach, taking hold in different contexts, is first to express a solution to the continuous problem and then to discretize the latter to obtain the desired approximation. That is, ``\emph{solve first and then discretize}''; consider, e.g., the use of operator-based Krylov methods for the solution of partial differential equations~\cite{MR3945243}, or the computation of continuous semigroups with error control~\cite{MR4379629}. 
{This case also encompasses the construction of exponential integrators based on the two-parameter Mittag--Leffler function~\cite{MR3350038}
\begin{equation}
\label{eq:mittag-leffler-definition}
E_{\alpha,\beta}(x) = \sum_{k=0}^{\infty} \frac{x^k}{\Gamma(\alpha k + \beta)},
\end{equation}
which, in the special case $\beta = 1$, reduces to the standard notation $E_{\alpha}(x)$.}

Our objective is to refer here to this second approach and to obtain an approximation of the solution of~\eqref{eq:fode} starting from its reformulation in terms of the so-called $\star$-product. 

The paper is organized as follows. In Section~\ref{sec:computingthesolution}, we provide a brief introduction to the $\star$-product formalism. We then explore its application in rewriting the solution of~\eqref{eq:fode}, first for scalar equations in Section~\ref{sec:solution_as_resolvents}, and subsequently for systems of equations in Section~\ref{sec:nonautonomous_systems}. In Section~\ref{sec:discretization}, we discuss the discretization procedure for the obtained formulation. Section~\ref{sec:numexamples} presents numerical experiments on benchmark problems to validate the approach. Finally, we conclude with a summary of findings and future research directions in Section~\ref{sec:conclusions}.

\section{Algebraic treatment of differential calculus with \texorpdfstring{$\star$-} ~products}\label{sec:computingthesolution}
\setcounter{section}{2} \setcounter{equation}{0} %

The main novelty driving the analytical and numerical progresses in the treatment of nonautonomous fractional differential equations is a product operation on bivariate distributions, denoted $\star$-product. We refer the reader to \cite{MR4191370,ryckebusch2023frechetlie} for a full exposition of this product, which we here only briefly present.

Let $I\subset \mathbb{R}$ be a compact and denote $\mathcal{C}^\infty(I^2)$ be the set of bivariate functions which are defined and smooth on an open neighborhood of $I^2$. The weak closure of this set, denoted $\overline{\mathcal{C}^\infty(I^2)}$, includes bivariate distribution-like objects such as $\delta(t-s)$, the Dirac delta distribution, its derivative $\delta'(t-s)$ and more generally all of its derivatives $\delta^{(i)}(t-s)$, $i\in\mathbb{N}$ as well as the Heaviside Theta function $\Theta(t-s)=1$ if $t\geq s$ and 0 otherwise. Contrary to the familiar Schwartz distributions, elements of $\overline{\mathcal{C}^\infty(I^2)}$ are not linear forms but endomorphisms of $\mathcal{C}^\infty(I^2)$ and the $\star$-product is their composition. To get an intuitive understanding of these objects, a good analogy is given by finite dimensional vector spaces: if smooth functions were column vectors, then Schwartz distributions such as $\delta(x)$ would be row vectors, endomorphisms of $\mathcal{C}^\infty(I^2)$ would be square matrices and the $\star$-product would be the matrix product. Of particular interest is the subset $\mathcal{D}\subset\overline{\mathcal{C}^\infty(I^2)}$ of endomorphisms of the form 
\[
d(t,s)=\widetilde{d}_{-1}(t,s)\Theta(t-s) + \sum_{i=0}^N \widetilde{d}_i(t,s)\delta^{(i)}(t-s),
\]
where all $\tilde{d}_i(t,s)\in\mathcal{C}^\infty(I^2)$, $i\geq -1$ are bivariate functions. We may succinctly state that $\mathcal{D}$ is the class of
all distributions which are linear combination of Heaviside theta functions and Dirac delta
derivatives with coefficients in $C^\infty(I^2)$. 
On $\mathcal{D}$ the $\star$-product is defined as
\begin{equation}\label{eq:starprod}
(d\star e)(t,s):=\int_{-\infty}^\infty d(t,\sigma)e(\sigma,s)d\sigma,\quad d,e\in\mathcal{D}.
\end{equation}
The unit for this product is $1_\star:=\delta(t-s)$. The set of $\star$-invertible elements of $\mathcal{D}$ is dense in $\mathcal{D}$ and forms a Fr\'echet-Lie group. The $\star$-product reduces to a convolution in the case where both $d(t,s)$ and $e(t,s)$ depend only on $t-s$ and otherwise differs from it. It also produces the Volterra composition, here denoted $\star_V$, \cite{volterralecons}: if $d(t,s):=\tilde{d}_{-1}(t,s)\Theta(t-s)$ and $e(t,s):=\tilde{e}_{-1}(t,s)\Theta(t-s)$, then $d\star e = \tilde{d}_{-1} \star_V \tilde{e}_{-1}$. Most of the properties of the $\star$-product are best understood as continuous analogs of those of the matrix product $(\mathsf{A}\mathsf{B})_{ij}=\sum_k \mathsf{A}_{ik} \mathsf{B}_{kj}$. In particular the $\star$-product is non-commutative but associative over $\mathcal{D}$.

In $\mathcal{D}$, $\star$-powers such as $\Theta^{\star n}$ and $(\delta')^{\star n}=\delta^{(n)}$ are well defined for $n\in\mathbb{Z}$~\cite{ryckebusch2023frechetlie} (Here and in what follows, we omit the $t,s$ arguments of the distributions and functions when they are clear from context). For instance, $\Theta^{\star-1}=\delta'$ and \begin{equation}\label{StarPowersTheta}
(\delta'^{\star -k})(t,s)=(\Theta^{\star k})(t,s)=\frac{(t-s)^{k-1}}{(k-1)!}\Theta(t-s),\quad k\in\mathbb{N}.
\end{equation}
Furthermore $\star$-multiplication by $\Theta$ produces integration \[(\Theta \star f)(t,s)=\int_{-\infty}^t f(\sigma,s)d\sigma,\] while the Cauchy formula for repeated integration of a function $\tilde{f}\in\mathcal{C}^\infty(I^2)$ is the statement $(\Theta^{\star k})\star \tilde{f}\Theta=\Theta\star(\Theta\star(\cdots \Theta\star(\Theta\star \tilde{f}\Theta)\cdots))$. These observations encourage us to define fractional integrals   from positive fractional $\star$-powers of~$\Theta$. Such fractional powers are accessible from the usual series definition over integer powers:
    \begin{lemma}\label{ThetaAlpha}
   Let $\alpha\in\mathbb{R}^+$. Then the endomorphism of $\overline{\mathcal{C}^\infty(I^2)}$, 
   \[
   \Theta^{\star \alpha}(t,s) := \sum_{k=0}^\infty\binom{\alpha}{k}\big(\Theta-\delta \big)^{\star k},
   \] 
    is well defined on every $I^2 \subset \mathbb{R}^2$ compact. Furthermore it evaluates to, $$\Theta^{\star \alpha}(t,s)  = \frac{(t-s)^{\alpha-1}}{\Gamma(\alpha)}\Theta(t-s).$$ This endomorphism is $\star$-invertible if and only if $\alpha\in\mathbb{N}$.
    \end{lemma}

    \begin{proof}
    Let $\alpha>0$, then consider the formal series definition of non-integer $\star$-powers of $\Theta$ from its integer powers,
        \begin{equation}\label{SeriesAlpha}
		\Theta^{\star \alpha} =\sum_{k=0}^\infty\binom{\alpha}{k}\big(\Theta-\delta \big)^{\star k}.
    \end{equation}
     Since $\delta$ is the $\star$-identity, $\Theta$ and $\delta$ $\star$-commute, $\delta^{\star (k-m)}=\delta$  and $\Theta^{\star 0}=\delta$, it holds that
   \begin{align*}
		\Theta^{\star \alpha}(t,s)&=\sum_{k=0}^\infty\binom{\alpha}{k}\sum_{m=0}^k\binom{k}{m} (-1)^{k-m}\,\Theta^{\star m}\\&=\sum_{k=0}^\infty\sum_{m=1}^k\binom{\alpha}{k}\binom{k}{m}(-1)^{k-m}\frac{(t-s)^{m-1}}{(m-1)!}\Theta(t-s)\\&\hspace{10mm}+\sum_{k=0}^\infty\binom{\alpha}{k}(-1)^k \delta(t-s).  
  \end{align*}
	Note that the assumption $\alpha>0$ is necessary to guarantee existence of the $\delta$ term. Evaluating the above expression gives, 
  \begin{align*}
		\Theta^{\star \alpha}(t,s)&=\sum _{k=0}^{\infty } (-1)^{k+1} k \binom{\alpha }{k} \, _1F_1(1-k;2;t-s)\Theta(t-s),\\&=\frac{(t-s)^{\alpha-1}}{\Gamma(\alpha)}\Theta(t-s).
	\end{align*}
Here $_1F_1(a,b,z)$ is a confluent hypergeometric function. Generalized hypergeometric functions~\cite[\S 16.2]{NIST:DLMF} \({}_pF_q\) are defined as the series:
\begin{equation}\label{pFq}
{}_pF_q(a_1, \dots, a_p; b_1, \dots, b_q; z) = \sum_{k=0}^{\infty} \frac{(a_1)_k (a_2)_k \cdots (a_p)_k}{(b_1)_k (b_2)_k \cdots (b_q)_k} \frac{z^k}{k!},
\end{equation}
where \((a)_k:=a (a+1) \cdots (a+k-1) = \Gamma(a+k)/\Gamma(a)\) is the Pochhammer symbol.\\

    We now turn to the $\star$-invertibility of $\Theta^{\star \alpha}$.
Should it be invertible, then its $\star$-inverse ought be the $\alpha$th fractional $\star$-power of $\delta'$, $(\delta')^{\star \alpha}$ since this is true for all integers $\alpha$, see Eq.~\eqref{StarPowersTheta}. Attempting to define $(\delta')^{\star \alpha}$ through a series in integer powers of $\delta'$ similarly to Eq.~(\ref{SeriesAlpha}) should lead to explicit forms for fractional derivatives of bivariate functions. The series does not converge for any $\alpha\in\mathbb{R}^+\backslash\mathbb{N}$ however, just as that of Eq.~\eqref{SeriesAlpha} does not converge for $\alpha<0$. Both of these facts indicate that $\Theta^{\star \alpha}$ is not $\star$-invertible for non-integer positive $\alpha$ values. In the following we will explicitly identify an element of its kernel. 
    \end{proof}
    
While the endomorphism $\Theta^{\star \alpha}$ is not $\star$-invertible hindering the definition of fractional derivatives, it nonetheless admits a pseudo-inverse for the $\star$-product from which stem both the Riemann--Liouville and Caputo fractional derivatives:
    \begin{lemma}\label{ThetaInverse}
   Let $a\in\mathbb{R}$, $\alpha\in\mathbb{R}^+$ and consider the operators $D_{\alpha}$ and $I$ defined through their action on any $f\in\mathcal{D}$ as
\begin{align}
\big(I \star f\big)(t,s)&:= \lim_{a\to-1} \int_{-\infty}^\infty \frac{(t-\sigma)^{a}}{\Gamma(1+a)}\Theta(t-\sigma) f(\sigma,s) d\sigma,\\
\big(D_\alpha \star f\big)(t,s) &:= \lim_{a\to-1} \int_{-\infty}^\infty \frac{(t-\sigma )^{a-\alpha }}{\Gamma (a-\alpha +1)} \Theta(t-\sigma) f(\sigma,s)d\sigma.\label{Daf}
\end{align}
Then,
\begin{align*}
&D_\alpha\star \Theta^{\star \alpha}=\Theta^{\star \alpha}\star D_\alpha = I,\\
&D_\alpha\star \Theta^{\star\alpha+\beta}=\Theta^{\star \beta},\quad \forall \beta\in\mathbb{R}^+,\\
&I\star \Theta^{\star \alpha}= \Theta^{\star \alpha}\star I=\Theta^{\star \alpha}.
\end{align*}
That is, $D_\alpha$ is the Drazin inverse \cite{KingDrazin} of $\Theta^{\star\alpha}$ with respect to the $\star$-product.
\end{lemma}

\begin{remark}
 Consider $\alpha\in\mathbb{R}^+\backslash\mathbb{N}$ and define
$$
X_\alpha(t,s):=\frac{(t-s)^{-\alpha }}{\Gamma
   (1-\alpha )}\Theta(t-s)-\Gamma (\alpha ) (t-s)^{1-\alpha }I(t,s),
$$
Then we have
$
\Theta^{\star \alpha}\star X_\alpha =0.
$
That is $\text{Im}(X_\alpha)\subset \text{Ker}(\Theta^{\star \alpha})$ for $\alpha\in\mathbb{R}^+\backslash\mathbb{N}$. The projector onto the kernel of $\Theta^{\star \alpha}$ is then $P:=\delta - D_\alpha\star \Theta^{\star \alpha}$. All the results above follow from direct calculations of the integrals. They offer an alternative point of view on the traditional approach developed by Gel'fand and Shilov~\cite{Gelfand1967} (and see also \cite{Podlubny1999}).%
    \end{remark}

Considering the objects $f\in\overline{\mathcal{C}^\infty(I^2)}$ for which \textit{the limit and integral of~\eqref{Daf} can be exchanged}, the Drazin inverse $D_\alpha=(\Theta^{\star \alpha})^\dagger$ is seen to act, by $\star$-product,  the same way as 
\begin{equation}\label{effectiveinverse}
\Theta^{\star-\alpha}:=\frac{(t-s)^{-1-\alpha}}{\Gamma(-\alpha)} \Theta(t-s).
\end{equation}
The $\star$-action of this is precisely what yields Riemann--Liouville integral with negative fractional order: let $\tilde{y}\in\mathcal{C}^\infty(I^2)$, then  
\begin{equation}\label{RiemannDraz}
\big(\Theta^{\star-\alpha}\star \tilde{y}\Theta\big)(t,s)=\,^\text{RL}\!_sI_t^{-\alpha}[\tilde{y}]\,\Theta(t-s)=:\,^\text{RL}\!_s\mathrm{D}_t^{\alpha}[\tilde{y}]\,\Theta(t-s),
\end{equation}
with $\,^\text{RL}\!_s\mathrm{D}_t^{\alpha}[\tilde{y}]$ the Riemann--Liouville derivative of $\tilde{y}(t)$ of order $\alpha$. 
This provides an algebraic explanation for the change $\alpha\to -\alpha$ postulated in the Riemann--Liouville approach to fractional derivatives from fractional integrals. However, this strategy is not algebraically valid \textit{in general}, since the $\star$-action of $\Theta^{\star-\alpha}$ does not always coincide with that of $D_\alpha$.
\begin{remark}
For $\alpha>0$, $\Theta^{\star-\alpha}$ is not a genuine $\star$-power of $\Theta$ since $\Theta$ has no $\star$-inverse when $\alpha>0$ is not an integer. Rather it should  be understood as a convenient abuse of notation to represent the action of the Drazin inverse~$D_\alpha$. In the following we will also use the similar notation $\Theta^{\star(1-\alpha)}:=(t-s)^{-\alpha}/\Gamma(1-\alpha)\,\Theta(t-s)$, which is not a $\star$-power of $\Theta$ whenever $\alpha>1$. 
\end{remark}

Let %
{$\tilde{y}^{(\alpha)}(t)$}
denote the Caputo derivative of order {$0 < \alpha < 1$} with respect to $t$ of the smooth function $\tilde{y}(t)$. This quantity also has an immediate algebraic meaning. Indeed, since $\Theta^{\star \alpha} = \Theta\star \Theta^{\star \alpha-1}$ and $\delta'=\Theta^{\star -1}$, it follows that $D_\alpha = D_{\alpha-1}\star \delta'$. Thus by Eq.~\eqref{RiemannDraz},
\begin{align*}
\prescript{\text{RL}}{s}{\mathrm{D}}_t^{\alpha}[\tilde{y}]\,\Theta = D_\alpha\star \tilde{y}\Theta =\Theta^{\star(1-\alpha)} \star \delta' \star \tilde{y}\Theta.
\end{align*}
Since $\delta'\star \tilde{y}\Theta = \tilde{y}'\Theta + \tilde{y} \delta$ (see, e.g., \cite{MR4191370}) then,
\begin{align*}
\prescript{\text{RL}}{s}{\mathrm{D}}_t^{\alpha}[\tilde{y}]\,\Theta&=\underbrace{\prescript{\text{RL}}{s}I_t^{1-\alpha}[\tilde{y}']}_{ %
{\tilde{y}^{(\alpha)}(t)}
}\,\Theta + \frac{(t-s)^{-\alpha}}{\Gamma(1-\alpha)}\,\tilde{y}(s) \Theta,
\end{align*}
from which we recover the Caputo derivative of $\tilde{y}$
\begin{equation}\label{CaputoDef}
{\tilde{y}^{(\alpha)}}(t)\Theta =\prescript{\text{RL}}{s}{\mathrm{D}}_t^{\alpha}[\tilde{y}]\,\Theta-\frac{(t-s)^{-\alpha}}{\Gamma(1-\alpha)}\,\tilde{y}(s) \Theta.
\end{equation}
Alternative presentations of the Caputo derivative that rely on $\tilde{y}^{(k)}$ instead of $\tilde{y}'$ are similarly obtained via $D_\alpha = D_{\alpha-k}\star \delta^{(k)}$.
    
\section{Solution to nonautonomous fractional differential equations as \texorpdfstring{$\star$-} ~resolvents}\label{sec:solution_as_resolvents}
\setcounter{section}{3} \setcounter{equation}{0}

Within the framework presented in Section~\ref{sec:computingthesolution}, we obtain a new expression of the solution of the fractional differential equation \eqref{eq:fode} by reformulating it explicitly as an equation in two variables $t>s$, $s$ being the time for the initial condition:
\begin{equation}\label{frac2var}
{\tilde{y}^{(\alpha)}(t,s)}
\Theta(t-s)= \tilde{f}(t)\,\tilde{y}(t,s)\Theta(t-s),\quad \tilde{y}(s,s) = \tilde{y}_s.
\end{equation}
\begin{theorem}\label{thm:star-solution}
Consider the homogeneous linear fractional differential equation \eqref{frac2var} with $\tilde{f}(t)$ a continuous function of $t$ and $\alpha\in\mathbb{R}^+$. Then the solution exists and is given~by
\begin{equation}\label{eq:star-solution}
    	\tilde{y}(t,s) \Theta = \Theta^{\star \alpha} \star (\delta -\tilde{f}\Theta^{\star\alpha})^{\star -1}\,\star \Theta^{\star(1-\alpha)}\,\tilde{y}_s.
\end{equation}
\end{theorem}

\begin{corollary}\label{inhomogeneouscorollary}
    Let $\alpha\in\mathbb{R}^+$ and consider the inhomogeneous linear fractional differential equation 
    \begin{equation}\label{frac2varinho}
{\tilde{y}^{(\alpha)}(t,s)}
\Theta(t-s)= \tilde{f}(t)\,\tilde{y}(t,s)\Theta(t-s)+\tilde{g}(t)\Theta(t-s),\quad \tilde{y}(s,s) = \tilde{y}_s
\end{equation}
    with $\tilde{f}$ and $\tilde{g}$ two continuous functions of $t$. Then the solution exists and is given~by
\begin{equation}\label{eq:star-solution2}
    	\tilde{y}(t,s) \Theta = \Theta^{\star \alpha} \star (\delta -\tilde{f}\Theta^{\star\alpha})^{\star -1}\,\star (\Theta^{\star(1-\alpha)}\,\tilde{y}_s+\tilde{g}\Theta).
\end{equation}
\end{corollary}

\begin{proof}
We first establish the Theorem. 
Let
$
G_\alpha :=D_\alpha\star \tilde{y}\,\Theta.
$
By Eq.~\eqref{CaputoDef} this is 
\begin{align*}
G_\alpha &= %
{\tilde{y}^{(\alpha)}(t)}\Theta
+ \frac{(t-s)^{-\alpha}}{\Gamma(1-\alpha)} \tilde{y}(s) \Theta%
 = %
 {\tilde{y}^{(\alpha)}(t)}\Theta + \Theta^{\star(1-\alpha)}\,\tilde{y}_s.
 \end{align*}
Furthermore, by Lemma~\ref{ThetaInverse}, 
	$$
	\tilde{f}\cdot \tilde{y}\Theta = \tilde{f}\cdot (\Theta^{\star\alpha}\star G_\alpha)=(\tilde{f} \cdot \Theta^{\star\alpha})\star G_\alpha, 
	$$
	where the last equality follows from the observation that $\tilde{f}$ depends only on $t$, and ``$\cdot$'' is the standard  product. Eq.~\eqref{frac2var} can now be recast as
	$$
	G_\alpha - \Theta^{\star(1-\alpha)} \tilde{y}_s = \tilde{f}\Theta^{\star\alpha}\star G_\alpha,
	$$
	that is 
	$$
	(\delta -\tilde{f}\Theta^{\star\alpha})\star G_\alpha  = \Theta^{\star(1-\alpha)} \tilde{y}_s,
	$$
	and thus
	$$
	G_\alpha =  (\delta -\tilde{f}\Theta^{\star\alpha})^{\star -1}\,\star \Theta^{\star(1-\alpha)}\tilde{y}_s.
	$$
	Given that $G_\alpha=D_\alpha\star y\Theta$ the solution is finally obtained as $\tilde{y}\Theta = \Theta^{\star \alpha}\star G_\alpha$ since $\Theta^{\star \alpha}\star D_\alpha=I$ by Lemma~\ref{ThetaInverse} and, for any smooth function $\tilde{y}$, we have $I\star \tilde{y}\Theta = \tilde{y} \Theta$. This gives
	\begin{align}
    \label{YSol}
	\tilde{y}\Theta &= \Theta^{\star \alpha} \star (\delta -\tilde{f}\Theta^{\star\alpha})^{\star -1}\,\star \Theta^{\star(1-\alpha)}\tilde{y}_s.
	\end{align}
We emphasize that this solution is valid in the nonautonomous situations where $\tilde{f}$ depends on $t$. 
Alternatively, since $\delta+(\delta - \tilde{f} \Theta^{\star \alpha})^{\star -1}\star \tilde{f}\Theta^{\star \alpha}=(\delta - \tilde{f} \Theta^{\star \alpha})^{\star -1}$ and $\Theta^{\star \alpha}\star \Theta^{\star(1-\alpha)}=\Theta$, 
\eqref{YSol} can be recast as 
$$\tilde{y}\Theta = \Theta + \Theta^{\star \alpha}\star (\delta - \tilde{f} \Theta^{\star \alpha})^{\star -1}\star \tilde{f}\Theta\,\tilde{y}_s.$$
For the inhomogeneous equation \eqref{frac2varinho}, following the same steps as above we find that $G_\alpha$ satisfies
$$
	(\delta -\tilde{f}\Theta^{\star\alpha})\star G_\alpha  = \Theta^{\star(1-\alpha)} \tilde{y}_s+\tilde{g}(t)\Theta.
$$
Together with $\tilde{y}\Theta = \Theta^{\star \alpha}\star G_\alpha$ this gets the result of  Corollary~\ref{inhomogeneouscorollary}.

\end{proof}
\begin{remark}
The solution \eqref{eq:star-solution} involves a $\star$-resolvent, $(\delta -\tilde{f}\Theta^{\star\alpha})^{\star -1}$, which is analytically available from the $\star$-Neumann series $\sum_{k=0}^\infty(\tilde{f}\Theta^{\star\alpha})^{\star k}$. This series is guaranteed to converge for $\tilde{f}\in\mathcal{C}^\infty(I^2)$ by standard arguments \cite{Giscard2015}. 
\end{remark}

\begin{example}\label{example:mlfunction}
Let $\tilde{f}(t) \equiv f$ be constant, $\alpha\in\mathbb{R}^+$ and consider the autonomous fractional differential equation for $t\geq s$,
$$
{\tilde{y}^{(\alpha)}(t)}
\Theta(t-s)= f \tilde{y}(t)\Theta(t-s).
$$
Given that $f\Theta^{\star \alpha}$ $\star$-commutes with $\Theta^{\star(1-\alpha)}$, the solution of \eqref{eq:star-solution} simplifies to 
$$
\tilde{y}(t,s)\Theta = \Theta \star (\delta -f\Theta^{\star\alpha})^{\star -1}\,\tilde{y}_s.$$ Furthermore for $k\geq 1$ an integer, $(f\Theta^{\star \alpha})^{\star k} = f^k\Theta^{\star k\alpha}$, so the Neumann series for the $\star$-resolvent yields, as expected,
\begin{align*}
\tilde{y}(t,s)\Theta &= \sum_{k=0}^{\infty} f^k\Theta^{\star k\alpha+1}\,\tilde{y}_s=\sum_{k=0}^\infty f^k \frac{(t-s)^{\alpha k}}{\Gamma(\alpha k+1)}\,\tilde{y}_s\Theta,\\
&= E_\alpha\big(f(t-s)^
{\alpha}\big)\,\tilde{y}_s\Theta,
\end{align*}
where $E_\alpha(\cdot)$ designates the Mittag-Leffler function~\eqref{eq:mittag-leffler-definition}
\end{example}

\begin{example}\label{example:nonautonomous}
Let $\alpha\in\mathbb{R}^+$ and consider the nonautonomous linear fractional differential equation 
$$
{\tilde{y}^{(\alpha)}(t)}
\Theta(t)= t\, \tilde{y}(t)\,\Theta(t),\quad \tilde{y}(0)=\tilde{y}_0.$$ 
Given that
$$
\Big(\Theta^{\star \alpha} \star(t\Theta^{\star\alpha})^{\star k}\star \Theta^{\star(1-\alpha)}\Big)(t,0)=\frac{(\alpha +1)^k \Gamma \left(k+\frac{1}{\alpha +1}+1\right) t^{(\alpha +1)   (k+1)}}{\Gamma \left(\frac{1}{\alpha +1}+1\right) \Gamma (\alpha  k+k+\alpha +2)},$$ 
exploiting the $\star$-Neumann series representation of the $\star$-resolvent appearing in the formal solution \eqref{eq:star-solution} we find, for $t\geq 0$,
$$
\tilde{y}(t) = \tilde{y}_0+\frac{\tilde{y}_0}{\Gamma \left(1+\frac{1}{\alpha +1}\right)}\sum_{k=0}^\infty \frac{(\alpha +1)^k\, \Gamma \left(k+\frac{1}{\alpha +1}+1\right) }{ \Gamma (\alpha  k+k+\alpha +2)}\,t^{(\alpha +1)
   (k+1)}.
$$
Once $\alpha$ is given, this can be further expressed in closed-form via generalized hypergeometric functions. For example $\alpha=1/2$ leads to
\begin{equation}\label{eq:example:nonautonomous:sol1}
    \tilde{y}(t) = \tilde{y}_0 \left[ \frac{4 t^{3/2} }{3 \sqrt{\pi
   }}\,
   _2F_2\left(1,\frac{4}{3};\frac{7}{6},\frac{3}{2};\frac{t^3}{3}\right)+\, _1F_1\left(\frac{5}{6};\frac{2}{3};\frac{t^3}{3}\right) \right],
\end{equation}
and $\alpha=1/3$ to
\begin{align}\label{eq:example:nonautonomous:sol2}
\tilde{y}(t) = &~\tilde{y}_0\left[
   _2F_2\left(\frac{7}{12},\frac{11}{12};\frac{1}{2},\frac{3}{4};\frac{t^4}{4}\right) \right. \nonumber\\
   &\qquad +\frac{63 \sqrt{3} t^{8/3} \Gamma \left(\frac{1}{3}\right) }{160 \pi }\,
   _3F_3\left(1,\frac{5}{4},\frac{19}{12};\frac{7}{6},\frac{17}{12},\frac{5}{3};\frac{t^4}
   {4}\right) \\
   &\qquad\left. +\frac{9 \sqrt{3} t^{4/3} \Gamma \left(\frac{2}{3}\right)}{8 \pi }\,
   _3F_3\left(\frac{11}{12},1,\frac{5}{4};\frac{5}{6},\frac{13}{12},\frac{4}{3};\frac{t^4}
   {4}\right) \right] \nonumber,
   \end{align}
while $\alpha=1$ gives $\tilde{y}(t) = e^{t^2/2}\,\tilde{y}_0$. We further remark that the results are valid for any $\alpha\in\mathbb{R}^+$, so for instance $\alpha=2$ correctly recovers \[\tilde{y}(t)=\,_0F_1(2/3, t^3/9)=(\tilde{y}_0/2) \Gamma(2/3) \left(3^{2/3}  \operatorname{Ai}(t)+3^{1/6} \operatorname{Bi}(t)\right),\] where $\operatorname{Ai}(\cdot)$ and $\operatorname{Bi}(\cdot)$ are the Airy functions~\cite[\S 9]{NIST:DLMF}.
   \end{example}

To advance our ``solve-then-discretize'' strategy numerically, we must introduce a suitable discretization for~\eqref{eq:star-solution} to obtain an approximate solution within a prescribed tolerance (see Section~\ref{sec:discretization}). However, our primary interest lies in solving systems of fractional differential equations. Therefore, before addressing the discretization process, the following Section~\ref{sec:nonautonomous_systems} outlines how the theoretical framework can be extended to accommodate this case. Notably, this scenario arises when solving nonautonomous linear partial differential equations with fractional time derivatives, where applying a longitudinal method of lines—i.e., a semi-discretization of the spatial variables—leads to the integration of a system of nonautonomous fractional differential equations.

\section{Nonautonomous systems of linear FDEs}\label{sec:nonautonomous_systems}
\setcounter{section}{4} \setcounter{equation}{0}

Let $\mathsf{M}(t), \mathsf{N}(t)\in \mathcal{C}^\infty(I)^{n\times n}$ be two matrix-valued functions that are smooth on an open neighborhood of a compact time interval $I$. We consider, on $I$, the fractional nonautonomous ordinary differential system 
\begin{equation}\label{Sys}
{\mathsf{U}^{(\alpha)}(t,s)}
= \mathsf{M}(t)\mathsf{U}(t,s)+\mathsf{N}(t),\quad \mathsf{U}(s,s)=\mathsf{I},
\end{equation}
where the Caputo derivative on the left hand side is understood to be taken on each entry of the solution matrix $\mathsf{U}$. The main result of the first part is that the formal solution of a nonautonomous linear fractional differential scalar equation with coefficient function $\tilde{f}(t)$ involves the $\star$-resolvent of $\tilde{f}\Theta^{\star \alpha}$. In the matrix case, 
the definition of $\star$-product \eqref{eq:starprod} generalizes by taking the usual matrix-matrix product inside the integral, i.e., 
\[  
(\mathsf{D}\star \mathsf{E})(t,s):=\int_{-\infty}^\infty \mathsf{D}(t,\sigma)\mathsf{E}(\sigma,s)d\sigma,\quad \mathsf{D},\mathsf{E}\in\mathcal{D}^{n \times n}.
\]
Then, the algebraic proof of Theorem~\ref{thm:star-solution} remains completely valid leading to a similar result, namely,
\begin{corollary}\label{cor:nonautonomous_system}
Consider the fractional nonautonomous system~\ref{Sys}. Then its solution is
\begin{equation}\label{MatrixSol}
\mathsf{U}(t,s)\Theta = 
(\Theta^{
        \star \alpha} \; \mathsf{I})\star \big(\mathsf{I}_\star -\mathsf{M}(t)\Theta^{\star\alpha}\big)^{\star -1} \star (\Theta^{\star(1-\alpha)}\;\mathsf{I}+\mathsf{N}(t)\Theta),
\end{equation}
with $\mathsf{I}_\star = \delta \mathsf{I}$ and $\mathsf{I}$ is the $n\times n$ identity matrix. 
\end{corollary}
The crucial advantage of this formal solution is that  
the $\star$-resolvent $(\mathsf{I}_\star -\mathsf{M}(t)\Theta^{\star\alpha})^{\star -1}$ at its core is amenable to powerful algebraic and combinatorial techniques, such as the path-sum theorem exploiting the structure inherent to the coefficient matrix $\mathsf{M}(t)$ \cite{Giscard2015,Giscard2013,Giscard2012}.

In its most general form, the method of path-sums stems from a fundamental property of walks on graphs: the existence and uniqueness of their factorization into simple paths and simple cycles, walks which do not visit any vertex more than once. This property entails that the formal series of all walks between any two vertices of a graph can be reduced to a fraction  over its simple cycles. This, in turn, implies that entries of resolvents can be represented as fractions over the simple cycles of an underlying graph.
For instance, in the case of interest here note that $(\mathsf{I}_\star -\mathsf{M}(t)\Theta^{\star \alpha})^{\star-1}=\sum_{k\geq0} (\mathsf{M}(t)\Theta^{\star \alpha})^{\star k}$. Seeing $(\mathsf{M}(t)\Theta^{\star \alpha})^{\star k}$ as the powers of an adjacency matrix, the Neumann series for the resolvent becomes a series of walks. Consequently{,} $(\mathsf{I}_\star -\mathsf{M}(t)\Theta^{\star \alpha})^{\star-1}_{ij}$ must admit an exact representation as a fraction over the simple cycles of a graph $\mathcal{G}$ whose adjacency matrix is $\mathsf{M}(t)\Theta^{\star \alpha}$. The approach is exactly the same as that presented in \cite{Giscard2015} for the nonautonomous, integer-order $\alpha=1$ differential equations{,}  so we refer the reader to this work for more details. Here{,}  we shall only present path-sums for fractional systems through an example.

\begin{example} Consider the following system of coupled linear fractional differential equations
\begin{equation}\label{SysExample}
{\mathsf{U}^{(\alpha)}(t)}
= \begin{pmatrix}1+t&-t\\1&0\end{pmatrix}\mathsf{U}(t),\quad \mathsf{U}(0)=\mathsf{I}.
\end{equation}
This is an instance of~\eqref{Sys} with $s=0$. Observe that in the ordinary case $\alpha=1$, the differential system being nonautononous, the solution $\mathsf{U}$ is the time-ordered exponential of $\mathsf{M}(t)=\begin{pmatrix}1+t&-t\\1&0\end{pmatrix}$, which is already difficult to obtain analytically. 

Returning to full generality $\alpha\in\mathbb{R}^+$, 
let $\mathsf{G}:=(\mathsf{I}_\star - \mathsf{M}(t)\Theta^{\star \alpha})^{\star -1}$. According to Eq.~\eqref{MatrixSol} the solution of this fractional differential system is then $$\mathsf{U} = (\Theta^{\star \alpha} \;\mathsf{I}) \star \mathsf{G}\star (\Theta^{\star(1-\alpha)}\;\mathsf{I}).$$
We obtain $\mathsf{G}$, and thence $\mathsf{U}$, with the path-sum theorem. To that end we begin by constructing the graph $\mathcal{G}$ with adjacency $\mathsf{M}(t)\Theta^{\star \alpha}$; see Figure~\ref{fig:pathsum}.
\begin{figure}[h!]
\vspace{-4mm}
    \centering
    \includegraphics[width=0.42\linewidth]{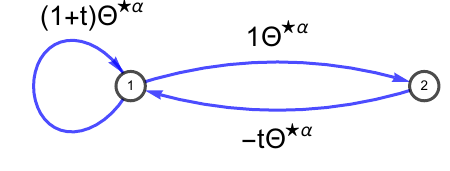}
    \vspace{-2mm}
    \caption{Graph $\mathcal{G}$ for the application of the path-sum theorem to Eq.~\eqref{SysExample}.}
    \label{fig:pathsum}
    \vspace{-5mm}
\end{figure}
\FloatBarrier

\noindent The quantity $\mathsf{G}_{11}$ is expressed in terms of simple cycles from $1$ to $1$ on this graph,
\begin{align*}
\mathsf{G}_{11}&=\Big(1_\star - \overbrace{(1+t)\Theta^{\star \alpha}}^{\text{Loop } 1\leftarrow 1} -\underbrace{(-t\Theta^{\star \alpha})\,\star\,(1\Theta^{\star \alpha})}_{\text{Backtrack }1\leftarrow 2\leftarrow 1}\Big)^{\star -1}.
\end{align*}
This simplifies to  $\mathsf{G}_{11}=\big(1_\star - \Theta^{\star \alpha})^{\star -1}\star (1_\star - t\Theta^{\star \alpha})^{\star -1}$. Each of these resolvents is accessible via its $\star$-Neumann series (see Examples~\ref{example:mlfunction} and \ref{example:nonautonomous}). We find (details in Appendix~\ref{AppA})
\begin{align*}
\mathsf{U}_{11}(t)&= \frac{1}{\Gamma \left(\frac{1}{\alpha +1}\right)}\sum_{k,m=0}^\infty\frac{(\alpha +1)^k\, \Gamma \left(k+\frac{1}{\alpha +1}\right)
   }{
   \Gamma (k+(k+m) \alpha +1)}t^{\alpha  (k+m)+k}.
\end{align*}
{While we could not find a closed-form expression for the above result in terms of Mittag Leffler or fractional trigonometric functions, we can recast it as an infinite series of such functions,
$$
\mathsf{U}_{11}(t)=\frac{1}{\Gamma \left(\frac{1}{\alpha +1}\right)}\sum_{k=0}^{\infty}(\alpha +1)^k \,\Gamma \Big(k+\frac{1}{\alpha +1}\Big) \,t^{(\alpha +1) k} E_{\alpha
   ,\,\alpha  k+k+1}\left(t^{\alpha }\right).
$$} 
Proceeding similarly for other matrix elements of $\mathsf{U}$, we obtain (see Appendix~\ref{AppA}),
\begin{align*}
\mathsf{U}_{21}&=\frac{1}{\Gamma \left(\frac{1}{\alpha
   +1}\right)}\sum_{k,m=0}^\infty \frac{(\alpha +1)^k\, \Gamma \left(k+\frac{1}{\alpha +1}\right)
   }{ \Gamma (k+(k+m+1) \alpha +1)}t^{\alpha  (k+m+1)+k},\\
   &={\frac{1}{\Gamma \left(\frac{1}{\alpha +1}\right)}\sum_{k=0}^\infty (\alpha +1)^k \Gamma \Big(k+\frac{1}{\alpha +1}\Big) t^{\alpha+(\alpha+1) k} \,E_{\alpha
   ,\,\alpha+\alpha  k+k +1}\left(t^{\alpha }\right)}\\
\mathsf{U}_{12}&=-\frac{1}{\Gamma
   \left(\frac{1}{\alpha +1}\right)}\sum_{k=1}^\infty\sum_{m=0}^\infty\frac{(\alpha +1)^k\, \Gamma \left(k+\frac{1}{\alpha +1}\right) }{ \Gamma (k+(k+m) \alpha +1)}t^{\alpha  (k+m)+k},\\
   &={-\mathsf{U}_{11}+E_{\alpha }\left(t^{\alpha }\right)},\\
   \mathsf{U}_{22}&=1-\frac{1}{\Gamma
   \left(\frac{1}{\alpha +1}\right)}\sum_{k=1}^\infty\sum_{m=0}^\infty \frac{(\alpha +1)^k\, \Gamma \left(k+\frac{1}{\alpha +1}\right) }{ \Gamma (k+(k+m+1) \alpha +1)}t^{\alpha  (k+m+1)+k},\\
   &={-\mathsf{U}_{21}+E_{\alpha }\left(t^{\alpha }\right)}.
\end{align*}
This solves the nonautonomous fractional differential system of \eqref{SysExample}.

\end{example}

 \section{Numerical solution}\label{sec:discretization}
\setcounter{section}{5} \setcounter{equation}{0} %
We aim to translate the solution expression~\eqref{eq:star-solution} into something amenable to computation. For this task we can adapt a technique from the ODE setting~\cite{pozza2023new}. Let us consider a sequence of the Legendre polynomials $\{P_k\}_k$, rescaled and shifted so that they are orthonormal over the bounded interval $I = [0,T]$ on which we have defined the FODE~\eqref{eq:fode}, that is
	\[
	\int_{I}  P_k(\tau) P_\ell(\tau) d\tau = \delta_{k,\ell} = \begin{cases}
		0,\quad \text{if }k\neq \ell\\
		1,\quad \text{if } k=\ell				
	\end{cases}.
	\]
	  In this way, $\{P_k\}_k$ is a basis for the space of smooth functions over $I$. Hence, we can expand any distributions of the form $f(t,s)= \widetilde{f}(t,s) \Theta(t-s)$ (which are \emph{piecewise smooth}) as follows
	\begin{align*}
	f(t,s) &= \sum_{k,\ell=0}^\infty f_{k,\ell} \, P_k(t) P_\ell(s), \; t \neq s, \; t,s \in I, {\text{ where }}\\
    f_{k,\ell} &= \iint_{I\times I} f(\tau,\sigma) P_k(\tau) P_\ell(\sigma) \; \mathrm{d} \tau \; \mathrm{d} \sigma.
	\end{align*}
	For a given positive integer $m$, the expansion above can be truncated and expressed in matrix format
	\begin{equation}\label{eq:the-truncated-expansion}
		f_m(t,s) = \sum_{k=0}^{m-1} \sum_{\ell=0}^{m-1} f_{k,\ell} \, {P}_k(t) {P}_\ell(s) = {\phi}_m(t)^\top \mathsf{F}_m \, {\phi}_m(s),
	\end{equation}    
    where the \emph{coefficient matrix} $\mathsf{F}_m$ and the \emph{basis vector} ${\phi}_m(t)$ are defined as
	\begin{equation}\label{eq:coeff:mtx}
		\mathsf{F}_m := \begin{bmatrix}
			f_{0,0} & f_{0,1} & \dots & f_{0,m-1}\\
			f_{1,0} & f_{1,1} & \dots & f_{1,m-1}\\
			\vdots & \vdots &  & \vdots\\
			f_{m-1,0} & f_{m-1,1} & \dots & f_{m-1,m-1}
		\end{bmatrix}, 
		\quad  {\phi}_m(t) :=
		\begin{bmatrix}
			P_0(t)\\
			P_1(t)\\
			\vdots\\
			P_{m-1}(t)
		\end{bmatrix}.
	\end{equation}

Let us define $\mathcal{D}_t$ as the subring of the $\star$-ring $(\mathcal{D}, \star, +)$ generated by the identity distribution $\delta$ and all distributions of the kind $f(t,s) = \tilde{f}(t)\Theta(t-s)$. Moreover, let $f, g, {q} \in \mathcal{D}_t$ be such that $q = f \star g$, and let $\mathsf{F}, \mathsf{G}, \mathsf{Q}$ be the respective infinite coefficient matrices ($m \rightarrow \infty$). As shown in \cite{pozza2023new,Poz24}, it follows that
\begin{equation}\label{eq:prodFG}
    \mathsf{Q} = \mathsf{F} \mathsf{G},
\end{equation}
meaning that the matrix product of the (infinite) matrices $\mathsf{F}, \mathsf{G}$ is well-defined and captures the $\star$-product operation.
In numerical computations, it is necessary to use the truncated coefficient matrices \eqref{eq:coeff:mtx}, with $m< \infty$. In this case, the matrix $\mathsf{Q}_m := \mathsf{F}_m \mathsf{G}_m$ provides only an approximation of $\mathsf{Q}$.
The quality of this approximation depends on the (numerical) band of the involved matrices; a discussion of the truncation errors can be found in~\cite{pozza2023new}.

    In \cite{Poz24}, we showed that $\mathcal{D}_t$ can be mapped onto a subalgebra of the usual matrix algebra. As a result, additional matrix algebra operations are also well-defined and align with corresponding $\star$-operations. Most importantly, the $\star$-inverse is mapped onto the usual inverse of a matrix. 
    By analytic continuation, for $\rho, \alpha \in \mathbb{R}^+$, the expression in Lemma~\ref{ThetaAlpha} can be generalized to
    \[ \Theta^{\star \alpha}(t,s) := \rho^\alpha \sum_{k=0}^\infty\binom{\alpha}{k}\left(\frac{\Theta}{\rho}-\delta \right)^{\star k}. \]  
    Then, the truncated series 
    $$S_j(t,s) := \rho^\alpha\sum_{k=0}^j\binom{\alpha}{k}\left(\frac{\Theta}{\rho}-\delta \right)^{\star k}{,} $$ 
    is in $\mathcal{D}_t$.
    Therefore, denoting with $\mathsf{H}_m${,} the coefficient matrix of $\Theta$, it holds
    \begin{align*}
        S_j(t,s) &= \rho^\alpha \sum_{k=0}^j\binom{\alpha}{k}\left(\frac{\Theta}{\rho}-\delta \right)^{\star k} \\
        &\approx \rho^\alpha \phi_m(t)^\top \left(\sum_{k=0}^j\binom{\alpha}{k}\left(\frac{\mathsf{H}_m}{\rho}-\mathsf{I}_m \right)^{k} \right) \phi_m(s),
    \end{align*}
    for $t,s \in I, t\neq s$.
    Note that for $\rho$ large enough, this last matrix series is always convergent for $j \rightarrow \infty$, obtaining
    \[ \Theta^\alpha(t,s) \approx \phi_m(t)^\top (\mathsf{H}_m)^\alpha \phi_m(s). \]
    In Table~\ref{tab:sta:mtx:alg} we summarize the other main results on the matrix representation of the $\star$-algebra.
   	\begin{table}[!ht]
		\centering
		\begin{tabular}{l|l}
\toprule
			$\star$-operations/objects & matrix operations/objects \\				\midrule
			$p = f \star g$              &  $\mathsf{P} = \mathsf{F} \mathsf{G}$  \\
			$p = f + g$           & $\mathsf{P} = \mathsf{F} + \mathsf{G}$    \\
			$1_{\star} = \delta$ &   $\mathsf{I} $, identity matrix         \\
			$f^{-\star}$                    & $\mathsf{F}^{-1}$, inverse    \\
			$(1_{\star}- f)^{-\star}$   &  $(\mathsf{I}-\mathsf{F})^{-1}$, resolvent  \\
            $\Theta$ & $\mathsf{H}$ \\
\bottomrule
		\end{tabular}
		\caption{Left: $\star$-algebra operations and related objects. Right: Corresponding matrix algebra operations and objects ($f, g, p$ are distributions from $\mathcal{D}_t$, and $\mathsf{F}, \mathsf{G}, \mathsf{P}$ are the corresponding coefficient matrices \eqref{eq:coeff:mtx}) for $m\rightarrow \infty$.}
		\label{tab:sta:mtx:alg}
	\end{table}

	Let us approximate the solution $\tilde{y}(t,s)$ of~\eqref{eq:star-solution} by combining the operations in Table~\ref{tab:sta:mtx:alg}, the truncated matrices \eqref{eq:coeff:mtx}, and the expression of Theorem~\ref{thm:star-solution}. First, we consider the truncated matrix $\mathsf{Y}_m$, i.e., the coefficient matrix of the solution:
	\[
	\tilde{y}(t,s) \approx \tilde{y}_m(t,s) := {\phi}_m(t)^\top \mathsf{Y}_m {\phi}_m(s).
	\]
	Then, we build the coefficient matrix $\mathsf{F}_m$ of the distribution $\widetilde{f}(t)\Theta^{\star \alpha}$. Finally, the solution~\eqref{eq:star-solution} rewrites, as in Theorem~\ref{thm:star-solution} 
	\[
	\tilde{y}(t,s) \approx {\phi}_m(t)^\top \mathsf{Y}_m {\phi}_m(s) \approx {\phi}_m(t)^\top \mathsf{H}_m^\alpha (\mathsf{I}_m - \mathsf{F}_m)^{-1} \mathsf{H}_m^{1-\alpha}{\phi}_m(s){,}
	\]
	where $\mathsf{I}_m$ is the identity matrix of size $m$. 
    As explained in \cite{pozza2023new}, because of the Gibbs phenomenon, the solution is computable only for $s=0$, which is exactly the time of the initial conditions; hence we can compute the solution of Equation~\eqref{eq:fode}.
    In a two-step formulation, this amounts to
	\begin{equation}\label{eq:the-problem-to-be-solved}
		\begin{cases}
			\text{Solve:}\; & (\mathsf{I}_m - \mathsf{F}_m) x = \mathsf{H}_m^{1-\alpha}{\phi}_m(0),\\
			\text{Approximate:}\; & y(t) \approx {\phi}_m(t)^\top \mathsf{H}_m^\alpha x, \quad \forall\,t \in I.
		\end{cases}
	\end{equation}
The elements of the vector $\mathsf{H}_m^\alpha x$ approximate the Legendre coefficients of the function $y(t)$. As shown in \cite{pozza2023new} for the case $\alpha = 1$ (ODE), the truncation error of $\mathsf{H}_m$ and $\mathsf{F}_m$ is amplified in $\mathsf{H}_m^\alpha x$. However, it tends to accumulate towards the last elements of the vector $\mathsf{H}_m^\alpha x$. Therefore, it can be eliminated by setting those elements to zero. While the analysis of the truncation error accumulation is beyond the scope of this paper, { the numerical experiments in Section~\ref{sec:numexamples} (in particular in Section~\ref{sec:exp:trunc:param})} confirm that this also happens for $0< \alpha < 1$. Therefore, it is possible to reduce the truncation error accumulation by erasing the last elements of $\mathsf{H}_m^\alpha x$.

Following the same ideas presented in the ODE case in \cite{Pozza2023a}, for the FODE system treated in Corollary~\ref{cor:nonautonomous_system}, we can also derive a linear algebra expression.
First we extend the operations and objects of Table~\ref{tab:sta:mtx:alg} to the case of $\mathcal{D}_t^{n \times n}$. As done in \cite{Poz24}, let $\mathsf{M}(t,s) = [f_{ij}(t,s)]_{i,j= 1}^n$ be an $n \times n$ matrix with elements $f_{ij}(t) \in \mathcal{D}_t^{n \times n}$. We can map each element $f_{ij}$ into the corresponding (truncated) coefficient matrix $\mathsf{F}_m^{(i,j)}$. Then, we get the following $nm \times nm$ matrix composed of the $m \times m$ blocks
    \begin{equation}\label{eq:coeff:mtx:mtx}
         \mathbf{A}_m := \left[\begin{array}{ccccccc}
					\boxed{\begin{array}{c}
							\mathsf{F}_m^{(1,1)}
					\end{array}} &  \dots & \boxed{\begin{array}{c}
							\mathsf{F}_m^{(1,n)}
					\end{array}}\\
					\vdots         &   \ddots    & \vdots\\
					\boxed{\begin{array}{c}
							\mathsf{F}_m^{(n,1)}
					\end{array}} & \dots & \boxed{\begin{array}{c}
							\mathsf{F}_m^{(n,n)}
					\end{array}}\\
				\end{array}\right].
\end{equation}
Now, given $\mathsf{A}(t,s), \mathsf{B}(t,s), \mathsf{C}(t,s) \in \mathcal{D}_t^{n \times n}$ so that $\mathsf{C}(t,s) = \mathsf{A}(t,s) \star \mathsf{B}(t,s)$, we denote with $\mathbf{A}, \mathbf{B}, \mathbf{C}$
the respective \emph{coefficient matrices} \eqref{eq:coeff:mtx:mtx} for $m \rightarrow \infty$.
Then, we have $\mathbf{C} = \mathbf{A} \mathbf{B}$.  Therefore, Table~\ref{tab:sta:mtx:alg} generalizes to Table~\ref{table:matrixOps}.
 	\begin{table}[!ht]
		\centering
		\begin{tabular}{l|l}
\toprule
			$\star$-operations/objects & matrix operations/objects \\				\midrule
			$\mathsf{C} = \mathsf{A} \star \mathsf{B}$              &  $\mathbf{C} =\mathbf{A} \mathbf{B}$  \\
			$\mathsf{C} = \mathsf{A} + \mathsf{B}$           & $\mathbf{C} = \mathbf{A} + \mathbf{B}$    \\
			$\mathsf{I}_{\star} = \delta(t-s)\mathsf{I}_n$ &   $\mathsf{I} $, identity matrix         \\
			$\mathsf{A}^{-\star}$                    & $\mathbf{A}^{-1}$, inverse    \\
			$(\mathsf{I}_{\star}- \mathsf{A})^{-\star}$   &  $(\mathsf{I}-\mathbf{A})^{-1}$, resolvent  \\
   \bottomrule
		\end{tabular}
		\caption{
        Left: $\star$-algebra operations and related objects. Right: Corresponding matrix algebra operations and objects ($\mathsf{A}, \mathsf{B}, \mathsf{C}$ are from $\mathcal{D}_t^{n \times n}$, and $\mathbf{A}, \mathbf{B}, \mathbf{C}$ are the corresponding block coefficient matrices, for $m \rightarrow \infty$.)}
		\label{table:matrixOps}
	\end{table}
 Then, the approximated solution of Equation~\eqref{Sys} on the interva $I$ is given by transforming the expression in Corollary~\ref{cor:nonautonomous_system} into the matrix formula:
\begin{equation*}%
    \mathsf{U}(t,s) \approx (\mathsf{I}_{n} \otimes \phi_m(t)^\top \mathsf{H}_m^\alpha)(\mathsf{I}_{nm}- \mathbf{M}_m)^{-1} (\mathsf{I}_n \otimes \mathsf{H}_m^{1-\alpha}  + \mathbf{N}_m) (\mathsf{I}_n \otimes \phi_m(s)), 
\end{equation*}
with $\mathbf{M}_m$ and $\mathbf{N}_m$ the (block) coefficient matrix of, respectively, $\mathsf{M}(t)\Theta^{\star \alpha}$ and $\mathsf{N}(t)\Theta$, and $\otimes$ the Kronecker product.

Now, let us consider the following non-autonomous homogeneous FODE
\begin{equation}\label{Sys:vec}
{\tilde{u}^{(\alpha)}(t)}
= \mathsf{M}(t) \tilde{u}(t),\quad \tilde{u}(0)=u_0, \quad t\geq 0,
\end{equation}
where $u_0, u(t)$ are now vectors. 
Then, under the same regularity assumption of the previous case, the FODE solution is approximated by the formula
\begin{equation}\label{eq:star:sol:vec}
    \tilde{u}(t) \approx (\mathsf{I}_{n} \otimes \phi_m(t)^\top \mathsf{H}_m^\alpha)(\mathsf{I}_{nm}- \mathbf{M}_m)^{-1} (u_0 \otimes \mathsf{H}_m^{1-\alpha}\phi_m(0)). 
\end{equation}
Essentially, the solution $\tilde{u}(t)$ is accessible by solving the linear system
\begin{equation}\label{eq:lin:sys}
(\mathsf{I}_{nm}- \mathbf{M}_m) x = u_0 \otimes \mathsf{H}_m^{1-\alpha}\phi_m(0). 
\end{equation}

Finally, the coefficient matrices $\mathsf{H}_m$ can be computed using the method described in \cite{pozza2023new}. Then $(\mathsf{H}_m)^\alpha$ can be obtained by established methods for the computation of matrix functions \cite{Hig08,HigLin13}. In our case, we used the Schur decomposition $H_m = U S U^H$ and the command \verb|S^alpha| in MATLAB R2024b or the command \verb|expm(logm(S)*alpha)| for MATLAB versions before the R2021b one. Note that we have chosen this approach because $m$ is not a large matrix; moreover, note that the Schur decomposition of $\mathsf{H}_m$ can be computed and stored once for all.
If the number of coefficients $m$ becomes large enough that the direct method with computational cost $O(m^3)$ is no longer scalable, an iterative approach can be employed to perform the matrix-vector product operations involving the matrices $\mathsf{H}^{\alpha}_m$ and $\mathsf{H}^{1-\alpha}_m$; see, e.g.,~\cite{MR4410753,MR4235307}. Computing the coefficient matrix of a distribution of the kind $\tilde{f}(t)\Theta^{\star \alpha}$, it is slightly more convoluted, we give the details in Apprendix~\ref{app:coeff}

Overall, the method can be summarized as follows
\begin{algo}[$\star$-method]
The solution of the system of FODEs \eqref{Sys:vec} is obtained by the following procedure.
\begin{algorithmic}[1]
    \State Compute the coefficient matrix $\mathbf{M}_m$.
    \State Solve the linear system \eqref{eq:lin:sys}.
    \State Compute the vector $y = (\mathsf{I}_{n} \otimes \phi_m(t)^\top \mathsf{H}_m^\alpha) x$.
    \State Set to zero the last elements of $y$.
    \State The solution at time $t \in I$ is given by $\phi_m(t)^\top y$. 
\end{algorithmic}
Note that $y$ contains the approximated Legendre coefficients of the truncated expansion of the solution.
\end{algo}
Generally speaking, the most costly part of the method is item 2., the solution of the linear system. For certain large, sparse (or structured) matrices, the solution can be more efficiently computed by a Krylov subspace approach.

\subsection{Matrix equation formulation and Krylov subspace approach}
First, assume that the system of FODEs \eqref{Sys:vec} is autonomous, i.e., $\mathsf{M}(t) \equiv \mathsf{M}$ is constant. Then the block coefficient matrix is $\mathbf{M}_m = \mathsf{M} \otimes \mathsf{H}_m^\alpha$.
As a consequence, the linear system \eqref{eq:lin:sys} can be transformed into the matrix equation
\begin{equation}\label{eq:matrix_equation_complete}
     \mathsf{X} - \mathsf{H}_m^\alpha \mathsf{X} \mathsf{M}^\top = \phi_m^{(\alpha,0)} u_0^\top,  
\end{equation}
with $x = \text{vec}(\mathsf{X})$ the vectorization of $\mathsf{X}$, that is, the $m \times n$ matrix $\mathsf{X}$ is transformed into the $mn$ vector $x$ by stacking the $\mathsf{X}$ columns,  and $$\phi_m^{(\alpha,0)} = \mathsf{H}_m^{(1-\alpha)} \phi_m(0).$$ We refer, e.g., to \cite{Sim16} for more details on matrix equations and their connection with linear systems. 
If $\mathsf{M}$ is a large and sparse (or structured) matrix, then, following \cite{Sim16}, we can use Arnoldi's algorithm~\cite[Alg.~6.1, p.160]{saad03} to build an orthogonal basis $\mathsf{V}_j$ for the Krylov subspace $\mathcal{K}_j(\mathsf{\bar M}, u_0)$, with $\mathsf{\bar M}$ the conjugate of $\mathsf{M}$. Then, by approximating the solution as $\mathsf{X} \approx \mathsf{Y} \mathsf{V}_j^H$, and considering the upper Hessenberg matrix $\mathsf{R}_j = \mathsf{V}_j^H \mathsf{\bar M} \mathsf{V}_j$ we get the reduced-size matrix equation
\begin{equation}\label{eq:matrix_eq_projected}
    \mathsf{Y} - \mathsf{H}_m^\alpha \mathsf{Y} \mathsf{R}_j^H = \phi_m^{(\alpha,0)}  e_1^\top \|u_0\|_2,  
\end{equation}
with $e_1$, the first vector of the canonical basis. 
This latter equation is of much smaller size, hence it can be solved by a direct approach, in our case, by means of the Schur decomposition \cite{Sim16}.

Consider now a non-autonomous system of FODEs where $\mathsf{M}(t) = \mathsf{K} + \mathsf{L} \tilde{f}(t)$, where $\mathsf{K}, \mathsf{L}$ are constant matrices and $\tilde{f}(t)$ is a scalar function. Then, following the same procedure as above, we obtain the matrix equation
\begin{equation}\label{eq:time_dependent_mex}
    \mathsf{X} - \mathsf{H}_m^\alpha \mathsf{X} \mathsf{K}^\top - \mathsf{F}_m \mathsf{X} \mathsf{L}^\top = \phi_m^{(\alpha,0)} u_0^\top,  
\end{equation} 
where $\mathsf{F}_m$ is the coefficient matrix of $\tilde{f}\Theta^{\star \alpha}$. In this situation, it is generally no longer possible to apply the Krylov subspace approach described above. However, it is often observed that equations such as this, where the right-hand side is a rank $1$ matrix, have a solution $\mathsf{X}$ with a numerically low rank (e.g., in \cite{Pozza2023a}). This is exploited to devise a cheaper Krylov-based solver as described in Appendix~\ref{sec:solution_alg_details} It is worth noting that, in general, reformulating the solution of fractional differential equations---particularly in the multidimensional case, or when spatial and temporal discretizations are treated simultaneously---as a suitable matrix equation is a topic that has been extensively studied in the literature~\cite{MR3498143,MR3995304}. A key advantage of the present case is that the matrix $\mathsf{H}_{m}^\alpha$ is typically low-dimensional and possesses a structure that can be effectively exploited.

\section{Numerical examples}\label{sec:numexamples}
\setcounter{section}{6} \setcounter{equation}{0} %

This section presents a series of numerical experiments designed to validate the solution formulation described in Section~\ref{sec:solution_as_resolvents}. It also outlines a framework for comparing the proposed approach with existing integrators in terms of both accuracy and computational cost, employing the techniques discussed in Section~\ref{sec:discretization}. 

We begin our analysis in Section~\ref{sec:scalar_cases} by evaluating the numerical solution of scalar problems. In particular, we consider an autonomous problem whose solution is given by the Mittag-Leffler function (Example~\ref{example:mlfunction}) and two closed-form solutions corresponding to scalar nonautonomous systems (Example~\ref{example:nonautonomous}). In Section~\ref{sec:fractional_schrodinger}, we extend our study to the numerical solution of the fractional Schr\"odinger equation, addressing both the autonomous case—where the associated Hamiltonian implies a time-independent potential—and the nonautonomous case, in which the potential is time-dependent.

The experiments were carried out on a Linux-based system equipped with an Intel\textsuperscript{\textregistered} Core\texttrademark\ i9-14900HX processor (24 cores, 48 threads, up to 5.8\,GHz), 64\,GB of RAM, and running \textsc{MATLAB} version 24.2.0.2740171 (R2024b Update 1) for the scalar test cases. The numerical solution of the fractional Schr\"odinger equation was executed on a single node of the Toeplitz cluster, located in the green data center of the University of Pisa. This node is equipped with dual Intel\textsuperscript{\textregistered} Xeon\textsuperscript{\textregistered} E5-2650 v4 processors (24 cores, 48 threads, 2.20\,GHz), 264\,GB of RAM, and it runs \textsc{MATLAB} version 9.10.0.1602886 (R2021a).
To compute Legendre basis expansion we employ the \texttt{Chebfun} package~\cite{chebfun} (v5.7.0). The code to reproduce the following example is available in the Git repository hosted on GitHub \href{https://github.com/Cirdans-Home/starfractional}{Cirdans-Home/starfractional}.

\subsection{\bf Scalar problems with closed form solutions}\label{sec:scalar_cases}

To validate the solution procedure introduced, we consider a fundamental test case: recovering the solution of the scalar fractional differential equation~\eqref{eq:fode} for a constant \(\lambda \in \mathbb{R}\). Specifically, we aim to reconstruct the solution of the linear homogeneous fractional ordinary differential equation with the Caputo fractional derivative of order \(\alpha \in (0,1)\) on the interval \([0, 2]\) discussed in Example~\ref{example:mlfunction}, and given by:
\[
f(t,\tilde{y}) = F \tilde{y}(t), \qquad F\,\in\,\mathbb{R}.
\]
The solution to this equation is expressed as:
\[
\tilde{y}(t) = \tilde{y}_0 E_\alpha(F t^\alpha),
\]
where \(E_\alpha(z)\) is the Mittag-Leffler function~\eqref{eq:mittag-leffler-definition}.
To validate the solution obtained using the $\star$-method, we compare the computed results with those obtained using the MATLAB function \texttt{ml} whose implementation is described in~\cite{MR3068601,MR3350038}.
\begin{figure}[htbp]
    \centering

    \subfigure[Comparison of the the solution computed through the \texttt{ml} function and the $\star$-Lanczos approach]{\input{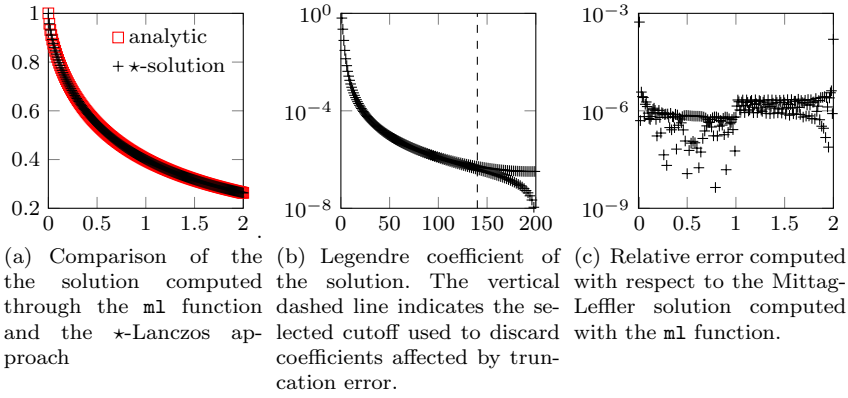}.}\hspace{0.6em}%
    \subfigure[Legendre coefficient of the solution. The vertical dashed line indicates the selected cutoff used to discard coefficients affected by truncation error.\label{fig:mittag-leffler-solution-legendrecoeff}]{\begin{tikzpicture}

\begin{axis}[%
width=0.217\columnwidth,
height=0.217\columnwidth,
at={(0\columnwidth,0\columnwidth)},
scale only axis,
xmin=0,
xmax=200,
ymode=log,
ymin=1e-08,
ymax=1,
yminorticks=true,
axis background/.style={fill=white},
]
\addplot [color=black, only marks, mark=+, mark options={solid, black}, forget plot]
  table[row sep=crcr]{%
1	0.639547393973909\\
2	0.223366283162605\\
3	0.0773166722551479\\
4	0.0301254825225075\\
5	0.0136740369890245\\
6	0.00714090644373295\\
7	0.00416673678745766\\
8	0.00264235923767331\\
9	0.00178360610409737\\
10	0.00126279309111692\\
11	0.000928224962412257\\
12	0.0007031226711427\\
13	0.000545974564830199\\
14	0.000432751207505432\\
15	0.00034910699893516\\
16	0.000285849919101296\\
17	0.000237175973464929\\
18	0.000199004664492524\\
19	0.000168723925790177\\
20	0.000144290803615551\\
21	0.000124449549776469\\
22	0.00010806284839429\\
23	9.45119471547073e-05\\
24	8.30970164447379e-05\\
25	7.35238405835443e-05\\
26	6.53179308155507e-05\\
27	5.83622380587032e-05\\
28	5.23041270902472e-05\\
29	4.71293611535615e-05\\
30	4.25535171021034e-05\\
31	3.86254808388226e-05\\
32	3.50997872901924e-05\\
33	3.20662194388268e-05\\
34	2.93016130412394e-05\\
35	2.69236775372342e-05\\
36	2.47217069378436e-05\\
37	2.28334902384347e-05\\
38	2.10547737624654e-05\\
39	1.95385263090538e-05\\
40	1.8083176718229e-05\\
41	1.68537274537409e-05\\
42	1.56488633785452e-05\\
43	1.46435500948679e-05\\
44	1.36351305816986e-05\\
45	1.28071412071377e-05\\
46	1.19545082194039e-05\\
47	1.12683849119801e-05\\
48	1.05405428201225e-05\\
49	9.9690886684839e-06\\
50	9.34212202965964e-06\\
51	8.86423612631294e-06\\
52	8.31948802489898e-06\\
53	7.91862648229525e-06\\
54	7.44139413573057e-06\\
55	7.10446035618863e-06\\
56	6.68304824988956e-06\\
57	6.39958215509961e-06\\
58	6.02460600176173e-06\\
59	5.78618448123857e-06\\
60	5.45005357819551e-06\\
61	5.24984213702404e-06\\
62	4.9463702256872e-06\\
63	4.77878420546788e-06\\
64	4.50289398608497e-06\\
65	4.36334009479899e-06\\
66	4.11083675325513e-06\\
67	3.99551407842172e-06\\
68	3.76291018574698e-06\\
69	3.66865571167845e-06\\
70	3.45303473034912e-06\\
71	3.37720245029881e-06\\
72	3.1761115194184e-06\\
73	3.11647713762642e-06\\
74	2.92784224132859e-06\\
75	2.88252752138368e-06\\
76	2.70458591563019e-06\\
77	2.67199823055563e-06\\
78	2.50324426298385e-06\\
79	2.48202799300326e-06\\
80	2.32116940126131e-06\\
81	2.3101666302806e-06\\
82	2.15608908421469e-06\\
83	2.15430763769112e-06\\
84	2.00604581510235e-06\\
85	2.01263313425162e-06\\
86	1.86934699847888e-06\\
87	1.88356868151447e-06\\
88	1.74452393092381e-06\\
89	1.76574602181642e-06\\
90	1.63029789578608e-06\\
91	1.65797220665646e-06\\
92	1.52555201463855e-06\\
93	1.55920390779561e-06\\
94	1.42930776838079e-06\\
95	1.46852594848574e-06\\
96	1.34070533844595e-06\\
97	1.38513329476364e-06\\
98	1.25898707991262e-06\\
99	1.30831588432703e-06\\
100	1.18348357664599e-06\\
101	1.23744580539536e-06\\
102	1.11360183049828e-06\\
103	1.17196642164821e-06\\
104	1.04881522479111e-06\\
105	1.11138311285135e-06\\
106	9.88654968689895e-07\\
107	1.05525537366283e-06\\
108	9.3270277703657e-07\\
109	1.00319004217054e-06\\
110	8.80584591828249e-07\\
111	9.54835486719433e-07\\
112	8.31965177745645e-07\\
113	9.09876594898956e-07\\
114	7.8654346162228e-07\\
115	8.68030446778123e-07\\
116	7.44048494796052e-07\\
117	8.29042568087439e-07\\
118	7.04235954317831e-07\\
119	7.92683675306181e-07\\
120	6.66885097244894e-07\\
121	7.58746843680221e-07\\
122	6.31796104523115e-07\\
123	7.27045037524125e-07\\
124	5.98787759699518e-07\\
125	6.97408949283612e-07\\
126	5.67695414764929e-07\\
127	6.69685107590936e-07\\
128	5.38369202430436e-07\\
129	6.43734216392819e-07\\
130	5.10672464079245e-07\\
131	6.19429692912376e-07\\
132	4.84480361224741e-07\\
133	5.96656380823516e-07\\
134	4.5967864758162e-07\\
135	5.75309414061814e-07\\
136	4.36162581834139e-07\\
137	5.55293211729575e-07\\
138	4.13835960436657e-07\\
139	5.36520590039974e-07\\
140	3.92610257714856e-07\\
141	5.18911973416919e-07\\
142	3.72403857309565e-07\\
143	5.02394696189659e-07\\
144	3.53141365151842e-07\\
145	4.86902381242694e-07\\
146	3.34752993663787e-07\\
147	4.72374389131825e-07\\
148	3.17174007358752e-07\\
149	4.58755327516527e-07\\
150	3.00344222656765e-07\\
151	4.45994615634192e-07\\
152	2.84207556775313e-07\\
153	4.34046097732315e-07\\
154	2.6871161641664e-07\\
155	4.22867699554918e-07\\
156	2.53807324899871e-07\\
157	4.12421124608975e-07\\
158	2.39448579337826e-07\\
159	4.026715853395e-07\\
160	2.25591937244893e-07\\
161	3.93587567233265e-07\\
162	2.12196325433792e-07\\
163	3.85140622298783e-07\\
164	1.99222771353827e-07\\
165	3.77305189329902e-07\\
166	1.8663414952474e-07\\
167	3.7005844012723e-07\\
168	1.74394945013972e-07\\
169	3.63380149597481e-07\\
170	1.6247102644661e-07\\
171	3.57252587961337e-07\\
172	1.50829429860457e-07\\
173	3.51660435449727e-07\\
174	1.39438148477146e-07\\
175	3.46590717915375e-07\\
176	1.28265926696266e-07\\
177	3.42032764002069e-07\\
178	1.17282056189056e-07\\
179	3.37978183308116e-07\\
180	1.06456169945882e-07\\
181	3.34420866629649e-07\\
182	9.57580346513918e-08\\
183	3.31357009080178e-07\\
184	8.51573342105503e-08\\
185	3.28785156997697e-07\\
186	7.46234448192194e-08\\
187	3.26706281574759e-07\\
188	6.412519681403e-08\\
189	3.25123880196654e-07\\
190	5.36306175743639e-08\\
191	3.24044109904695e-07\\
192	4.3106654800808e-08\\
193	3.23475956324551e-07\\
194	3.2518870295846e-08\\
195	3.23431441624982e-07\\
196	2.18311007640001e-08\\
197	3.23925880977246e-07\\
198	1.10050782976428e-08\\
199	3.24978190017241e-07\\
200	0\\
};
\addplot [color=black, dashed, forget plot]
  table[row sep=crcr]{%
140	1e-08\\
140	1\\
};
\end{axis}
\end{tikzpicture}%
}\hspace{0.6em}%
    \subfigure[Relative error computed with respect to the Mittag-Leffler solution computed with the \texttt{ml} function.\label{fig:mittag-leffler-solution-relative-error}]{\begin{tikzpicture}

\begin{axis}[%
width=0.217\columnwidth,
height=0.217\columnwidth,
at={(0\columnwidth,0\columnwidth)},
xtick={1,50,100,150,200},
xticklabels={0,0.5,1,1.5,2},
scale only axis,
xmin=0,
xmax=200,
ymode=log,
ymin=1e-09,
ymax=0.001,
yminorticks=true,
axis background/.style={fill=white},
]
\addplot [color=black, only marks, mark=+, mark options={solid, black}, forget plot]
  table[row sep=crcr]{%
1	0.000538441147488133\\
2	4.98593234499312e-07\\
3	3.77693784557557e-06\\
4	2.72461858381064e-06\\
5	2.56514067103483e-06\\
6	6.62218122738428e-07\\
7	1.58740289680569e-06\\
8	1.87031728914436e-06\\
9	1.35177860807028e-06\\
10	8.48608558623976e-07\\
11	5.81258482352235e-07\\
12	5.40847559571971e-07\\
13	6.68703699754056e-07\\
14	8.92153666599337e-07\\
15	1.11337766155157e-06\\
16	1.20388404331955e-06\\
17	1.03365362923043e-06\\
18	5.43410258727991e-07\\
19	1.68364376841706e-07\\
20	8.15932478372737e-07\\
21	1.03408341723889e-06\\
22	6.2477289734224e-07\\
23	2.0712914605276e-07\\
24	8.77054436240933e-07\\
25	8.23268649217556e-07\\
26	4.41113148167733e-08\\
27	7.65702969346135e-07\\
28	7.94310881771597e-07\\
29	2.0862348033372e-08\\
30	7.99506573541171e-07\\
31	6.30297905807943e-07\\
32	3.05060649700446e-07\\
33	8.37955246240819e-07\\
34	2.6038562763348e-07\\
35	6.61612886211612e-07\\
36	6.4138745276275e-07\\
37	3.03623190239153e-07\\
38	7.84065097351413e-07\\
39	6.62530113047238e-08\\
40	7.43795508080215e-07\\
41	3.59894476562044e-07\\
42	6.02365664818536e-07\\
43	5.54852496231066e-07\\
44	4.2666717835353e-07\\
45	6.63977251802647e-07\\
46	2.58339437717593e-07\\
47	7.12132605626717e-07\\
48	1.17764562910178e-07\\
49	7.2355516299842e-07\\
50	1.16788388141307e-08\\
51	7.16944710891078e-07\\
52	6.02770585144593e-08\\
53	7.04698932917925e-07\\
54	1.01240705489783e-07\\
55	6.93814691268273e-07\\
56	1.14780944908029e-07\\
57	6.87151611703637e-07\\
58	1.03910328198194e-07\\
59	6.84528541455208e-07\\
60	7.0923723990945e-08\\
61	6.83538659801863e-07\\
62	1.76943523596725e-08\\
63	6.80146635839388e-07\\
64	5.3801189823185e-08\\
65	6.69191393726558e-07\\
66	1.40880241877988e-07\\
67	6.44910561079968e-07\\
68	2.39586322188131e-07\\
69	6.01573857892612e-07\\
70	3.44266056674978e-07\\
71	5.34257443474252e-07\\
72	4.47395212153437e-07\\
73	4.39735699042259e-07\\
74	5.39754099040912e-07\\
75	3.17397357410758e-07\\
76	6.11016459759929e-07\\
77	1.7004129700311e-07\\
78	6.50763941432105e-07\\
79	4.36596768816352e-09\\
80	6.49851691939978e-07\\
81	1.6902882865742e-07\\
82	6.01963133138422e-07\\
83	3.36268869935511e-07\\
84	5.0510321950455e-07\\
85	4.81504712215747e-07\\
86	3.6273386300366e-07\\
87	5.88802385869092e-07\\
88	1.84255988472347e-07\\
89	6.44398792806563e-07\\
90	1.5379235885324e-08\\
91	6.38978568664233e-07\\
92	2.17013297802511e-07\\
93	5.69559346510566e-07\\
94	3.99474921299934e-07\\
95	4.40592515746775e-07\\
96	5.42340377743855e-07\\
97	2.63990189824459e-07\\
98	6.28822938185729e-07\\
99	5.79755603428072e-08\\
100	6.48290066081987e-07\\
101	1.39998911932921e-06\\
102	2.15062343516989e-06\\
103	1.90211583701024e-06\\
104	1.06466936842171e-06\\
105	1.03517489310987e-06\\
106	1.86109470927575e-06\\
107	2.1560788894347e-06\\
108	1.41979364038248e-06\\
109	8.82770329669155e-07\\
110	1.44611470904151e-06\\
111	2.16118236202329e-06\\
112	1.81782071110615e-06\\
113	9.88832316415164e-07\\
114	1.07416588907177e-06\\
115	1.93059234189075e-06\\
116	2.10055622355623e-06\\
117	1.28669103420901e-06\\
118	8.73789492969188e-07\\
119	1.57100099425395e-06\\
120	2.18189012233034e-06\\
121	1.64685987350068e-06\\
122	8.82720642014507e-07\\
123	1.21580793801632e-06\\
124	2.07048351554422e-06\\
125	1.94844959265129e-06\\
126	1.05344497821063e-06\\
127	9.60668086312488e-07\\
128	1.8407639885819e-06\\
129	2.12792603385118e-06\\
130	1.29757007638807e-06\\
131	8.35112301336781e-07\\
132	1.58251685447616e-06\\
133	2.18663979934455e-06\\
134	1.53359865649333e-06\\
135	8.12813399913748e-07\\
136	1.3623270058309e-06\\
137	2.16805458314169e-06\\
138	1.7124803533082e-06\\
139	8.40783935192978e-07\\
140	1.21202671231073e-06\\
141	2.12767377711365e-06\\
142	1.81641034840495e-06\\
143	8.66227800210484e-07\\
144	1.13998745325364e-06\\
145	2.11155236741798e-06\\
146	1.84067604095347e-06\\
147	8.52574437003419e-07\\
148	1.1529410079971e-06\\
149	2.1439060586644e-06\\
150	1.77112633828229e-06\\
151	7.91617975044237e-07\\
152	1.27612039260987e-06\\
153	2.21071040543003e-06\\
154	1.57068764646515e-06\\
155	7.28935734663703e-07\\
156	1.55293779817122e-06\\
157	2.22219084866423e-06\\
158	1.20555762922825e-06\\
159	8.1072029326736e-07\\
160	1.97488339268287e-06\\
161	1.97309770597877e-06\\
162	7.79790020304056e-07\\
163	1.27451072419245e-06\\
164	2.27425386240819e-06\\
165	1.27911757044736e-06\\
166	7.74701011076549e-07\\
167	2.07995581903311e-06\\
168	1.80352935666092e-06\\
169	6.42265642198099e-07\\
170	1.73475144167889e-06\\
171	2.11163139117603e-06\\
172	6.89048458768373e-07\\
173	1.50706345693084e-06\\
174	2.22712688432339e-06\\
175	6.96734880885876e-07\\
176	1.51703660802632e-06\\
177	2.19459180881095e-06\\
178	5.81233434436185e-07\\
179	1.84332971787033e-06\\
180	1.86452055207814e-06\\
181	5.61592127157818e-07\\
182	2.38760793010154e-06\\
183	9.61531186050116e-07\\
184	1.39162867271143e-06\\
185	2.06558008308418e-06\\
186	5.31624749127527e-07\\
187	2.55948782616378e-06\\
188	3.35501350041819e-07\\
189	2.60600977308519e-06\\
190	2.90658829243147e-07\\
191	2.72735205335271e-06\\
192	1.50766264418116e-07\\
193	2.55436464118054e-06\\
194	1.31485490946204e-06\\
195	7.6149787467289e-08\\
196	2.50102590487558e-06\\
197	3.59645116022609e-06\\
198	4.13518852232523e-06\\
199	8.2137493479207e-07\\
200	0.000158896379285126\\
};
\end{axis}
\end{tikzpicture}%
}

    \caption{Numerical solution with the $\star$-Lanczos approach on the autonomous scalar test equation from Example~\ref{example:mlfunction} and comparison with the true solution expressed in term of the Mittag-Leffler function ($\tilde{y}_0 = 1$, $F = -1$, $\alpha = 0.7$ and $M = 200$ basis function used).}
    \label{fig:mitag-leffler-solution}
\end{figure}
Fig.~\ref{fig:mitag-leffler-solution} contains the solution computed with the $\star$-approach for $\tilde{y}_0 = 1$, $F = -1$ and $\alpha = 0.7$, in particular, panel~\ref{fig:mittag-leffler-solution-legendrecoeff} shows the Legendre basis coefficients corresponding to the application of the discrete equivalent of the $\star$-resolvent to the Legendre coefficient expansion of the basis function. In the experiment, we observe the behavior described in Section~\ref{sec:discretization}, where the final coefficients are affected by truncation error due to the finite number of basis functions used in the series expansion. To reconstruct the solution, we select a subset of coefficients in the Legendre basis that are not impacted by truncation—specifically, $k=140$—thereby obtaining the relative error shown in the third panel, Fig~\ref{fig:mittag-leffler-solution-relative-error}. This relative error is consistent with the expected order of magnitude corresponding to the truncated coefficients.

The other test case we want to consider is the one based on Example~\ref{example:nonautonomous} where we consider a nonautonomous FDE with source term \[f(t,y) = t y(t), \qquad t \geq 0,\]  on the interval \([0, 2]\), by exploiting the $\star$-formalism we have obtained an expression of the solution in terms of generalized hypergeometric functions~\eqref{pFq} that we can then use to validate what was obtained through the solution procedure described in Section~\ref{sec:discretization}. Specifically, we select $\tilde{y}_0 = $1 and $\alpha = \nicefrac{1}{2}$ to recover the solution~\eqref{eq:example:nonautonomous:sol1}, and $\tilde{y}_0 = 1$ and $\alpha = \nicefrac{1}{3}$ to recover the solution~\eqref{eq:example:nonautonomous:sol2}. 
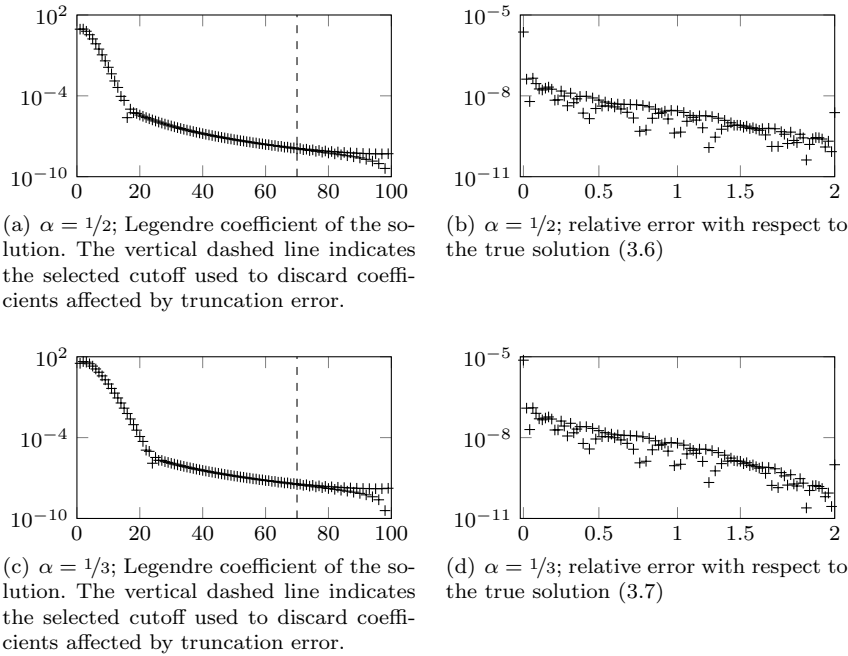
\begin{figure}[htbp]
    \centering

    \subfigure[$\alpha = \nicefrac{1}{2}$; Legendre coefficient of the solution. The vertical dashed line indicates the selected cutoff used to discard coefficients affected by truncation error.]{\begin{tikzpicture}

\begin{axis}[%
width=0.35\columnwidth,
height=0.18\columnwidth,
at={(0\columnwidth,0\columnwidth)},
scale only axis,
xmin=0,
xmax=100,
ymode=log,
ymin=1e-10,
ymax=100,
yminorticks=true,
axis background/.style={fill=white},
]
\addplot [color=black, only marks, mark=+, mark options={solid, black}, forget plot]
  table[row sep=crcr]{%
1	8.85096864682702\\
2	8.99713443809286\\
3	6.20376139041134\\
4	3.43066281519333\\
5	1.67097616876192\\
6	0.729202359092187\\
7	0.294948534736707\\
8	0.110433156885693\\
9	0.0394155887583715\\
10	0.0131076682445695\\
11	0.00430759555354362\\
12	0.00126649340439496\\
13	0.000417012799311065\\
14	9.12408492120857e-05\\
15	4.55055772856379e-05\\
16	2.40574053537973e-06\\
17	1.0287944652112e-05\\
18	5.64054113541276e-06\\
19	4.96094551419046e-06\\
20	3.76418315052394e-06\\
21	3.03760715068911e-06\\
22	2.4401278289224e-06\\
23	1.98917466306122e-06\\
24	1.63345991418738e-06\\
25	1.35366775920849e-06\\
26	1.12949731564432e-06\\
27	9.49852496714516e-07\\
28	8.0317590516264e-07\\
29	6.84053147714603e-07\\
30	5.85026415449101e-07\\
31	5.03812318467167e-07\\
32	4.35104536890318e-07\\
33	3.7840058507406e-07\\
34	3.29560259546797e-07\\
35	2.89150715412346e-07\\
36	2.53667771266914e-07\\
37	2.24360606217948e-07\\
38	1.98063146872114e-07\\
39	1.76488724130022e-07\\
40	1.56634980733019e-07\\
41	1.40554116190459e-07\\
42	1.25301113143819e-07\\
43	1.13194122034161e-07\\
44	1.01277608476815e-07\\
45	9.20934341226679e-08\\
46	8.2629991556153e-08\\
47	7.56292475343612e-08\\
48	6.79910477534086e-08\\
49	6.26456020507574e-08\\
50	5.63794264916584e-08\\
51	5.23068483051679e-08\\
52	4.70807060262239e-08\\
53	4.4000684854656e-08\\
54	3.95679919972869e-08\\
55	3.72728662260115e-08\\
56	3.34479992844032e-08\\
57	3.17825313095142e-08\\
58	2.84239020565895e-08\\
59	2.72711902971399e-08\\
60	2.42693482870334e-08\\
61	2.35409040880446e-08\\
62	2.08100526334822e-08\\
63	2.04387209648983e-08\\
64	1.7910598760774e-08\\
65	1.78454717404242e-08\\
66	1.54648652518767e-08\\
67	1.5667588034229e-08\\
68	1.33890062045352e-08\\
69	1.38310904918602e-08\\
70	1.16162514307191e-08\\
71	1.22771113385199e-08\\
72	1.00930176527882e-08\\
73	1.09585531978704e-08\\
74	8.77597282913893e-09\\
75	9.83755169928157e-09\\
76	7.629798567032e-09\\
77	8.88355636747232e-09\\
78	6.62546381601678e-09\\
79	8.07186632336127e-09\\
80	5.73889325650015e-09\\
81	7.38250600432634e-09\\
82	4.94990071203226e-09\\
83	6.79936492513569e-09\\
84	4.24135124287142e-09\\
85	6.30955794398418e-09\\
86	3.59845754005851e-09\\
87	5.90294480118414e-09\\
88	3.00818733172394e-09\\
89	5.57181974416633e-09\\
90	2.45870563837242e-09\\
91	5.31076008010091e-09\\
92	1.93877881930589e-09\\
93	5.11675547367301e-09\\
94	1.43690364767482e-09\\
95	4.98997006123129e-09\\
96	9.39381221723058e-10\\
97	4.93759469612124e-09\\
98	4.09484830644725e-10\\
99	5.06951842959225e-09\\
100	0\\
};
\addplot [color=black, dashed, forget plot]
  table[row sep=crcr]{%
70	1e-10\\
70	100\\
};
\end{axis}
\end{tikzpicture}%
}\hspace{1em}
    \subfigure[$\alpha = \nicefrac{1}{2}$; relative error with respect to the true solution~\eqref{eq:example:nonautonomous:sol1}]{\begin{tikzpicture}

\begin{axis}[%
width=0.35\columnwidth,
height=0.18\columnwidth,
at={(0\columnwidth,0\columnwidth)},
scale only axis,
xtick={1,25,50,70,100},
xticklabels={0,0.5,1,1.5,2},
xmin=0,
xmax=100,
ymode=log,
ymin=1e-11,
ymax=1e-05,
yminorticks=true,
axis background/.style={fill=white},
]
\addplot [color=black, only marks, mark=+, mark options={solid, black}, forget plot]
  table[row sep=crcr]{%
1	2.33036741281722e-06\\
2	4.15769686219902e-08\\
3	6.15867829160718e-09\\
4	4.51383564073284e-08\\
5	2.84844325473095e-08\\
6	1.79758646247129e-08\\
7	1.65052698612712e-08\\
8	1.92067802857453e-08\\
9	2.11162097932405e-08\\
10	1.76411538292919e-08\\
11	6.88777817371812e-09\\
12	7.13148360089337e-09\\
13	1.50845675233408e-08\\
14	9.83902321956345e-09\\
15	4.27112836501504e-09\\
16	1.25764015190975e-08\\
17	5.56693690414217e-09\\
18	7.74394371374395e-09\\
19	9.51196246606056e-09\\
20	2.29091893620445e-09\\
21	9.6342395006329e-09\\
22	1.42276805658902e-09\\
23	8.34831808783253e-09\\
24	3.38652948319011e-09\\
25	6.94954383994645e-09\\
26	4.16061444220746e-09\\
27	5.90197994530279e-09\\
28	4.22607662692606e-09\\
29	5.26828022021199e-09\\
30	3.87430141280204e-09\\
31	4.95913064987495e-09\\
32	3.25601931465816e-09\\
33	4.84242138293586e-09\\
34	2.44820322394147e-09\\
35	4.78414682434686e-09\\
36	1.50427832163241e-09\\
37	4.66566191010194e-09\\
38	4.83214901009199e-10\\
39	4.39487049784071e-09\\
40	5.38369228005538e-10\\
41	3.91741275759029e-09\\
42	1.46932914874982e-09\\
43	3.22431340229286e-09\\
44	2.21676686458907e-09\\
45	2.35394313529435e-09\\
46	2.70482447903646e-09\\
47	1.38393021122127e-09\\
48	2.89157627603695e-09\\
49	4.14915368761078e-10\\
50	2.77961723700565e-09\\
51	4.51780416898369e-10\\
52	2.41518000108188e-09\\
53	1.13512292002987e-09\\
54	1.877457723678e-09\\
55	1.58980913232696e-09\\
56	1.26000517836299e-09\\
57	1.81086842295888e-09\\
58	6.51685012775827e-10\\
59	1.82907295417599e-09\\
60	1.20712841249146e-10\\
61	1.69814511573783e-09\\
62	2.9275848473317e-10\\
63	1.48038953001381e-09\\
64	5.76914278631759e-10\\
65	1.23326554140576e-09\\
66	7.41520860026527e-10\\
67	1.00143698922462e-09\\
68	8.08941743120743e-10\\
69	8.1338106627614e-10\\
70	8.04492449710809e-10\\
71	6.82783309651827e-10\\
72	7.49449054419989e-10\\
73	6.11254824457386e-10\\
74	6.56212692974412e-10\\
75	5.91250660499066e-10\\
76	5.26944084127771e-10\\
77	6.05850496980949e-10\\
78	3.54864612286803e-10\\
79	6.25525603663972e-10\\
80	1.32063600672711e-10\\
81	6.01661576390697e-10\\
82	1.3223117688998e-10\\
83	4.67419939956105e-10\\
84	3.76074674119773e-10\\
85	1.70270409457569e-10\\
86	4.50824720891217e-10\\
87	2.20841725713455e-10\\
88	1.84006673268361e-10\\
89	3.7619298282637e-10\\
90	2.75857174727601e-10\\
91	4.17144749032511e-11\\
92	1.56961551992332e-10\\
93	2.57501814240493e-10\\
94	2.83454173132526e-10\\
95	2.77863361452554e-10\\
96	2.50303212202561e-10\\
97	1.26352092928978e-10\\
98	2.07542131219456e-10\\
99	8.50436349137271e-11\\
100	2.38568112607045e-09\\
};
\end{axis}

\end{tikzpicture}%
}

    \subfigure[$\alpha = \nicefrac{1}{3}$; Legendre coefficient of the solution. The vertical dashed line indicates the selected cutoff used to discard coefficients affected by truncation error.]{\begin{tikzpicture}

\begin{axis}[%
width=0.35\columnwidth,
height=0.18\columnwidth,
at={(0\columnwidth,0\columnwidth)},
scale only axis,
xmin=0,
xmax=100,
ymode=log,
ymin=1e-10,
ymax=100,
yminorticks=true,
axis background/.style={fill=white},
]
\addplot [color=black, only marks, mark=+, mark options={solid, black}, forget plot]
  table[row sep=crcr]{%
1	32.5218249431901\\
2	43.964591203725\\
3	39.8911097787172\\
4	30.1610775718156\\
5	20.2617688050091\\
6	12.4355695851229\\
7	7.10310033504339\\
8	3.81774242022565\\
9	1.94880294093447\\
10	0.950078496033075\\
11	0.445082860129756\\
12	0.200880435855294\\
13	0.0878339347291637\\
14	0.0371789270024098\\
15	0.0153632111902325\\
16	0.00614368574389138\\
17	0.00243048525281689\\
18	0.000915575632295064\\
19	0.000356559637584932\\
20	0.000121019855356117\\
21	5.19135797183207e-05\\
22	1.17465998320985e-05\\
23	9.57468739659509e-06\\
24	1.2494269330225e-06\\
25	3.36561922764151e-06\\
26	2.0645505282175e-06\\
27	2.02074636215981e-06\\
28	1.6445606141649e-06\\
29	1.44611513292451e-06\\
30	1.24218492936196e-06\\
31	1.08520313740835e-06\\
32	9.46304329726068e-07\\
33	8.32450122609936e-07\\
34	7.32032677861057e-07\\
35	6.49141495543205e-07\\
36	5.74619761630326e-07\\
37	5.1343727455701e-07\\
38	4.57019978217448e-07\\
39	4.11277638622826e-07\\
40	3.67785538391309e-07\\
41	3.33225164692391e-07\\
42	2.99115143705126e-07\\
43	2.72794156244927e-07\\
44	2.4559243303132e-07\\
45	2.25441638111095e-07\\
46	2.0338988235736e-07\\
47	1.87930982372626e-07\\
48	1.69758825743169e-07\\
49	1.57920614400015e-07\\
50	1.42696142806205e-07\\
51	1.33692565017156e-07\\
52	1.20722200388504e-07\\
53	1.13970142701064e-07\\
54	1.02730507632854e-07\\
55	9.77933483678949e-08\\
56	8.78839948982218e-08\\
57	8.44324393386753e-08\\
58	7.55427286694012e-08\\
59	7.33271044655396e-08\\
60	6.5212614272527e-08\\
61	6.40432078588142e-08\\
62	5.6508590741841e-08\\
63	5.62414299155085e-08\\
64	4.91279200344764e-08\\
65	4.96545448193229e-08\\
66	4.28306303224183e-08\\
67	4.40705539405894e-08\\
68	3.74249505517501e-08\\
69	3.93202306021999e-08\\
70	3.27564341538765e-08\\
71	3.52676733919626e-08\\
72	2.86997049335828e-08\\
73	3.18033026273735e-08\\
74	2.51521882354847e-08\\
75	2.88383932736204e-08\\
76	2.20292513646141e-08\\
77	2.63009469772942e-08\\
78	1.92603504601314e-08\\
79	2.41325829292411e-08\\
80	1.67861428148936e-08\\
81	2.22861185309892e-08\\
82	1.45559443483162e-08\\
83	2.07237109075667e-08\\
84	1.25257828269847e-08\\
85	1.94156575018842e-08\\
86	1.06565826781082e-08\\
87	1.83394966035035e-08\\
88	8.9124804475247e-09\\
89	1.74798357838102e-08\\
90	7.25890515962276e-09\\
91	1.6828815922293e-08\\
92	5.65986205392454e-09\\
93	1.63886001367328e-08\\
94	4.07269208951306e-09\\
95	1.61779666116836e-08\\
96	2.4322053433067e-09\\
97	1.62948724191681e-08\\
98	3.83749438185622e-10\\
99	1.74167654323708e-08\\
100	0\\
};
\addplot [color=black, dashed, forget plot]
  table[row sep=crcr]{%
70	1e-10\\
70	100\\
};
\end{axis}
\end{tikzpicture}%
}\hspace{1em}
    \subfigure[$\alpha = \nicefrac{1}{3}$; relative error with respect to the true solution~\eqref{eq:example:nonautonomous:sol2}]{\begin{tikzpicture}

\begin{axis}[%
width=0.35\columnwidth,
height=0.18\columnwidth,
at={(0\columnwidth,0\columnwidth)},
xtick={1,25,50,70,100},
xticklabels={0,0.5,1,1.5,2},
scale only axis,
xmin=0,
xmax=100,
ymode=log,
ymin=1e-11,
ymax=1e-05,
yminorticks=true,
axis background/.style={fill=white},
]
\addplot [color=black, only marks, mark=+, mark options={solid, black}, forget plot]
  table[row sep=crcr]{%
1	7.46453318356544e-06\\
2	1.23952641086449e-07\\
3	1.99586624077156e-08\\
4	1.27776195238195e-07\\
5	7.96438436987466e-08\\
6	4.98701715726887e-08\\
7	4.57229198000329e-08\\
8	5.32588576382365e-08\\
9	5.85594791380919e-08\\
10	4.89120387958834e-08\\
11	1.91939360806776e-08\\
12	1.94072271458968e-08\\
13	4.12324915757863e-08\\
14	2.69111927610493e-08\\
15	1.14925419051853e-08\\
16	3.40106848148706e-08\\
17	1.50696840211113e-08\\
18	2.0744826247383e-08\\
19	2.54621401972642e-08\\
20	6.06350273115808e-09\\
21	2.5565788909922e-08\\
22	3.78433952672073e-09\\
23	2.19701002745381e-08\\
24	8.8882095765921e-09\\
25	1.8138423813873e-08\\
26	1.08162114148542e-08\\
27	1.52754399009195e-08\\
28	1.08834254125281e-08\\
29	1.351683149495e-08\\
30	9.88023935547515e-09\\
31	1.26065191483885e-08\\
32	8.2168985657181e-09\\
33	1.21881858896027e-08\\
34	6.10711893134286e-09\\
35	1.19131683044175e-08\\
36	3.70064347143646e-09\\
37	1.14842775924665e-08\\
38	1.15712513904642e-09\\
39	1.06827530325504e-08\\
40	1.32645200716436e-09\\
41	9.39258335893622e-09\\
42	3.52649263120209e-09\\
43	7.61520566773924e-09\\
44	5.22809412570724e-09\\
45	5.4658411895945e-09\\
46	6.26934727819634e-09\\
47	3.14816059096633e-09\\
48	6.5805476342258e-09\\
49	9.07427420473145e-10\\
50	6.20164926829301e-09\\
51	1.0246212337799e-09\\
52	5.27275836781661e-09\\
53	2.47873037028307e-09\\
54	4.00087004268711e-09\\
55	3.37884594818517e-09\\
56	2.61136390810928e-09\\
57	3.7446736577054e-09\\
58	1.30338790421489e-09\\
59	3.67327083082037e-09\\
60	2.16268796281993e-10\\
61	3.30371006085031e-09\\
62	5.81443462418531e-10\\
63	2.78162842515059e-09\\
64	1.08548058632046e-09\\
65	2.23055350833852e-09\\
66	1.33613489452057e-09\\
67	1.73686811645236e-09\\
68	1.39480887609917e-09\\
69	1.3475705776038e-09\\
70	1.32392505577323e-09\\
71	1.07629908779497e-09\\
72	1.17296571142781e-09\\
73	9.13709951690347e-10\\
74	9.73365794039948e-10\\
75	8.35052795853211e-10\\
76	7.38155271784553e-10\\
77	8.05470731302463e-10\\
78	4.69150904559688e-10\\
79	7.78924437533958e-10\\
80	1.68502822617589e-10\\
81	6.98276993488039e-10\\
82	1.39180042395984e-10\\
83	5.03559273938023e-10\\
84	3.7983162145994e-10\\
85	1.73048852400177e-10\\
86	4.20449498869174e-10\\
87	1.9049005557574e-10\\
88	1.60130554045749e-10\\
89	3.00729042794065e-10\\
90	2.04519642480365e-10\\
91	2.46606614092153e-11\\
92	1.09002691498173e-10\\
93	1.64014593977713e-10\\
94	1.67553615545234e-10\\
95	1.53194671781863e-10\\
96	1.29580769122964e-10\\
97	6.40953430532404e-11\\
98	8.78246051483471e-11\\
99	2.78510295429195e-11\\
100	9.87239274397921e-10\\
};
\end{axis}
\end{tikzpicture}%
}
    
    \caption{Numerical solution with the $\star$-Lanczos approach on the autonomous scalar test equation from Example~\ref{example:nonautonomous} and comparison with the true solution expressed in term of the generalized Hypergeometric function ($\tilde{y}_0 = 1$, $\alpha = \nicefrac{1}{2}$--first row--$\tilde{y}_0 = 1$, $\alpha = \nicefrac{1}{3}$; $M = 100$ basis function used).}
    \label{fig:hyper_geom_solution}
\end{figure}
The results are presented in Figure~\ref{fig:hyper_geom_solution}, where we observe a behavior similar to that of the autonomous case. In particular, as predicted by the theory, it is necessary to consider a number of Legendre coefficients smaller than the computed total, $M=100$, due to truncation error---we retain $k=70$ coefficients for both cases. Furthermore, the error behavior remains robust across different values of the fractional derivative order.

Another notable phenomenon observed in both cases (Figs.~\ref{fig:mitag-leffler-solution} and~\ref{fig:hyper_geom_solution}), in agreement with the general theory of FDEs, is that the region with the highest error is near the initial condition. This aligns with the well-established fact that the solution of a fractional differential equation exhibits the least regularity near the integration boundary. In this region, we are approximating a function with limited regularity using a truncated polynomial-type series expansion, which is inherently analytic, leading to a localized increase in error. A potential direction for future research is to explore alternative basis functions beyond Legendre polynomials, specifically non-polynomial bases that can more effectively capture the loss of regularity in the solution.

Having completed this construction, we can focus on solving more challenging fractional ODE systems in the next Section~\ref{sec:fractional_schrodinger}. 

\subsubsection{{Truncation error and parameter robustness}}\label{sec:exp:trunc:param}

{We briefly test the behavior of the $\star$-product approach with respect to parameters $m$ (the number of polynomial terms in the expansion~\eqref{eq:the-truncated-expansion}), $k \leq m-1$ (the cutoff used to eliminate coefficients affected by the truncation error), the value of the fractional derivative $\alpha$, and the length of the integration interval~$[0,T]$ for varying $T$.}

{Let us consider again the test case introduced in Example~\ref{example:nonautonomous}, for which an explicit solution~\eqref{eq:example:nonautonomous:sol1} is available for $\alpha = 0.5$ on the interval $[0,2]$. This setting allows us to assess the accuracy of the method in terms of the relative error as the number of terms $m$ in the series expansion increases.
\begin{figure}[htbp]
    \centering
    \definecolor{mycolor1}{rgb}{0.06600,0.44300,0.74500}%
\definecolor{mycolor2}{rgb}{0.12941,0.12941,0.12941}%
\begin{tikzpicture}

\begin{axis}[%
width=0.628\columnwidth,
height=0.271\columnwidth,
at={(0\columnwidth,0\columnwidth)},
scale only axis,
xmin=200,
xmax=4000,
xtick={200,1000,2000,4000},
xticklabels={$2\times 10^2$,$10^3$,$2\times 10^3$,$4 \times 10^3$},
xlabel style={font=\color{mycolor2}},
xlabel={m values (log-scale)},
ymode=log,
xmode=log,
log basis y=10,
log basis x=2,
ymin=1e-11,
ymax=1e-07,
yminorticks=true,
ylabel style={font=\color{mycolor2}},
ylabel={Maximum Error},
axis background/.style={fill=white},
xmajorgrids,
ymajorgrids,
yminorgrids
]
\addplot [color=black, mark=o, mark options={solid, black}, forget plot]
  table[row sep=crcr]{%
200	1.00448714235313e-07\\
400	1.25118120308859e-08\\
600	3.70278829820462e-09\\
800	1.56117518913427e-09\\
1000	7.99040833516648e-10\\
1200	4.62311966488695e-10\\
1400	2.91074719895495e-10\\
1600	1.95019112145802e-10\\
1800	1.36949119098264e-10\\
2000	9.98188199398077e-11\\
2200    7.492940e-11 \\
2400    5.773693e-11 \\
2600    4.545431e-11  \\
2800    3.637979e-11 \\
3000    2.959943e-11 \\
3200    2.433875e-11 \\
3400    2.036416e-11 \\
3600    1.713829e-11  \\
3800    1.453326e-11  \\
4000    1.249045e-11 \\
};
\end{axis}
\end{tikzpicture}%
    \caption{Error with respect to the exact solution for the test case in Example~\ref{example:nonautonomous} with respect to the number of terms $m$ in the series expansion.}
    \label{fig:m-robustness}
\end{figure}
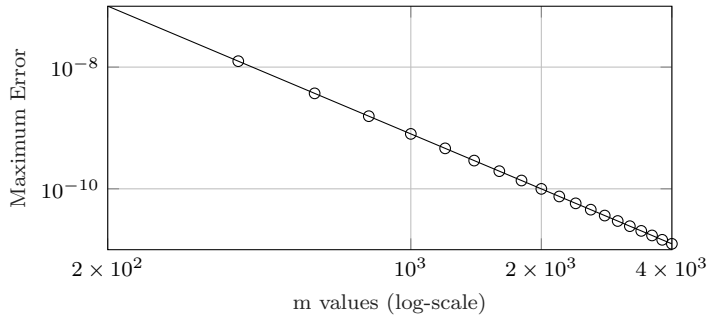
The results are reported in Fig.~\ref{fig:m-robustness}, where $m$ ranges from $200$ to $4000$ with increments of $200$. As the solution lacks regularity at $t=0$, its expansion in Legendre polynomials is expected to exhibit only algebraic convergence with respect to $m$.}

{In Fig.~\ref{fig:robustness-alpha} we consider the solution of Eq.~\eqref{eq:fode} with $\tilde{f}(t) = -t$ and $y_0 = 1$. We vary $\alpha \in [0.1,0.9]$ to investigate the robustness of the method with respect to the fractional order. We always employ $m = 10^3$ and compute the best value of $k$ against a reference solution obtained through the Fractional-BDF2 (FBDF2) method implemented in the \texttt{flmm} code~\cite{garrappa2018numerical,MR3327641} with step $h = 10^{-5}$ (the expected accuracy is of the order $10^{-10}$).}
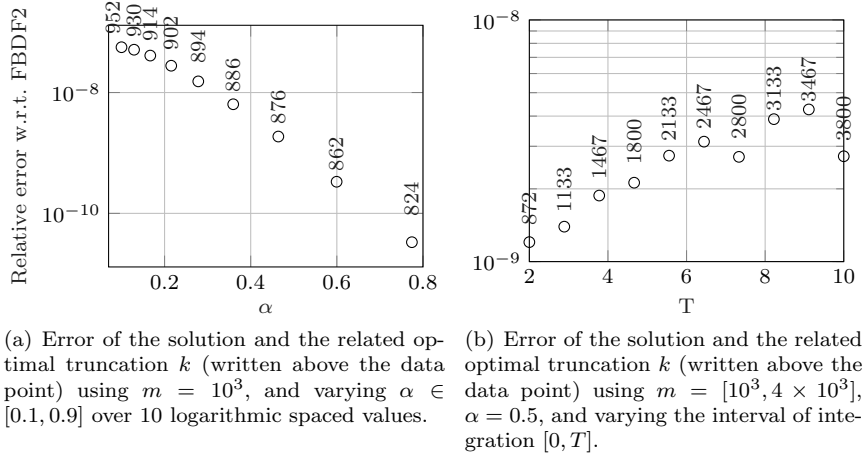
\begin{figure}[htbp]
    \centering

    \subfigure[{Error of the solution and the related optimal truncation $k$ (written above the data point) using $m = 10^3$, and varying $\alpha \in [0.1,0.9]$ over $10$ logarithmic spaced values.\label{fig:robustness-alpha}}]{\definecolor{mycolor1}{rgb}{0.06600,0.44300,0.74500}%
\definecolor{mycolor2}{rgb}{0.12941,0.12941,0.12941}%
\begin{tikzpicture}

\begin{axis}[%
width=0.35\columnwidth,
height=0.269\columnwidth,
at={(0\columnwidth,0\columnwidth)},
scale only axis,
clip=false,
xmin=0.07,
xmax=0.80,
xlabel style={font=\color{mycolor2}},
xlabel={$\alpha$},
ymode=log,
ymin=1.29772246862898e-11,
ymax=1.29772246862898e-07,
yminorticks=true,
ylabel style={font=\color{mycolor2}},
ylabel={Relative error w.r.t. FBDF2},
axis background/.style={fill=white},
xmajorgrids,
ymajorgrids,
yminorgrids
]
\addplot [color=mycolor1, only marks, mark=o, mark options={solid, black}, forget plot]
  table[row sep=crcr]{%
0.1	5.62415341128198e-08\\
0.129154966501488	5.11395084057398e-08\\
0.166810053720006	4.08902880535851e-08\\
0.215443469003188	2.77903282743619e-08\\
0.278255940220712	1.53021464813716e-08\\
0.359381366380463	6.39887428881911e-09\\
0.464158883361278	1.86951387703971e-09\\
0.599484250318941	3.33782335104857e-10\\
0.774263682681127	3.33342777685814e-11\\
};
\node[right, align=left, inner sep=0, rotate=90, font=\color{mycolor2}]
at (axis cs:0.08,9e-08) {952};
\node[right, align=left, inner sep=0, rotate=90, font=\color{mycolor2}]
at (axis cs:0.129,8e-08) {930};
\node[right, align=left, inner sep=0, rotate=90, font=\color{mycolor2}]
at (axis cs:0.167,7e-08) {914};
\node[right, align=left, inner sep=0, rotate=90, font=\color{mycolor2}]
at (axis cs:0.215,5e-08) {902};
\node[right, align=left, inner sep=0, rotate=90, font=\color{mycolor2}]
at (axis cs:0.278,3e-08) {894};
\node[right, align=left, inner sep=0, rotate=90, font=\color{mycolor2}]
at (axis cs:0.359,1e-08) {886};
\node[right, align=left, inner sep=0, rotate=90, font=\color{mycolor2}]
at (axis cs:0.464,3e-09) {876};
\node[right, align=left, inner sep=0, rotate=90, font=\color{mycolor2}]
at (axis cs:0.599,5e-10) {862};
\node[right, align=left, inner sep=0, rotate=90, font=\color{mycolor2}]
at (axis cs:0.774,9e-11) {824};
\end{axis}
\end{tikzpicture}%
}\hfil
    \subfigure[{Error of the solution and the related optimal truncation $k$ (written above the data point) using $m = [10^3,4\times 10^3]$, $\alpha=0.5$, and varying the interval of integration $[0,T]$.\label{fig:robustness-T}}]{\definecolor{mycolor1}{rgb}{0.06600,0.44300,0.74500}%
\definecolor{mycolor2}{rgb}{0.12941,0.12941,0.12941}%
\begin{tikzpicture}

\begin{axis}[%
width=0.35\columnwidth,
height=0.269\columnwidth,
at={(0\columnwidth,0\columnwidth)},
scale only axis,
clip=false,
xmin=2,
xmax=10,
xlabel style={font=\color{mycolor2}},
xlabel={T},
ymode=log,
ymin=1e-09,
ymax=1e-08,
yminorticks=true,
axis background/.style={fill=white},
xmajorgrids,
ymajorgrids,
yminorgrids
]
\addplot [color=mycolor1, only marks, mark=o, mark options={solid, black}, forget plot]
  table[row sep=crcr]{%
2	1.20396248526333e-09\\
2.88888888888889	1.39491750837428e-09\\
3.77777777777778	1.87639452709322e-09\\
4.66666666666667	2.12060784984898e-09\\
5.55555555555556	2.74383290198685e-09\\
6.44444444444444	3.13560641067999e-09\\
7.33333333333333	2.71277878724503e-09\\
8.22222222222222	3.88630548828212e-09\\
9.11111111111111	4.26185214702203e-09\\
10	2.728133201449e-09\\
};
\node[right, align=left, inner sep=0, rotate=90, font=\color{mycolor2}]
at (axis cs:2,1.40396248526333e-09) {872};
\node[right, align=left, inner sep=0, rotate=90, font=\color{mycolor2}]
at (axis cs:2.889,1.59491750837428e-09) {1133};
\node[right, align=left, inner sep=0, rotate=90, font=\color{mycolor2}]
at (axis cs:3.778,2.17639452709322e-09) {1467};
\node[right, align=left, inner sep=0, rotate=90, font=\color{mycolor2}]
at (axis cs:4.667,2.52060784984898e-09) {1800};
\node[right, align=left, inner sep=0, rotate=90, font=\color{mycolor2}]
at (axis cs:5.556,3.14383290198685e-09) {2133};
\node[right, align=left, inner sep=0, rotate=90, font=\color{mycolor2}]
at (axis cs:6.444,3.53560641067999e-09) {2467};
\node[right, align=left, inner sep=0, rotate=90, font=\color{mycolor2}]
at (axis cs:7.333,3.21277878724503e-09) {2800};
\node[right, align=left, inner sep=0, rotate=90, font=\color{mycolor2}]
at (axis cs:8.222,4.18630548828212e-09) {3133};
\node[right, align=left, inner sep=0, rotate=90, font=\color{mycolor2}]
at (axis cs:9.111,4.56185214702203e-09) {3467};
\node[right, align=left, inner sep=0, rotate=90, font=\color{mycolor2}]
at (axis cs:10,3.128133201449e-09) {3800};
\end{axis}
\end{tikzpicture}%
}
    
    \caption{{Robustness of the solution with respect to different parameters defining the method (expansion terms $m$ and truncation parameter $k$) and the one defining the problem (order of the fractional derivative $\alpha$, and length of the time interval $T$).}}
    \label{fig:robustness}
\end{figure}
{We observe that the error decreases with growing $\alpha$. This corresponds to the well-known regularity behavior of the solution with respect to the value of $\alpha$: lower values of $\alpha$ always correspond to a solution with lower regularity; hence, we typically need more terms in the polynomial expansion to achieve a better approximation.}

{In Fig.~\ref{fig:robustness-T}, we consider the same test problem with $\alpha = 0.5$ while expanding the integration interval $[0,T]$. Specifically, we select ten consecutive values of $T$ uniformly spaced between $2$ and $10$. Correspondingly, for each value of $T$, we increase $m$ by choosing ten uniformly spaced values in the range $m \in [10^3,\, 4 \times 10^3]$. Again, we compute the best value of $k$ against the FBDF2 reference solution with step $h = 10^{-5}$.}

{For all intervals $[0,T]$, the method preserves an error comparable to that observed on $[0,2]$, namely of order $10^{-9}$. As the integration interval increases, larger values of $(k,m)$, as well as higher truncation levels and more terms in the series expansion, are required to maintain this level of accuracy. This example suggests that the required $m$ value is linearly scaling with the size of the interval.}

\subsection{\bf Time-Fractional Schr\"{o}dinger equation}\label{sec:fractional_schrodinger}

The time-fractional Schr\"{o}dinger equation~\cite{PhysRevE.62.3135,MR1755089,PhysRevE.66.056108,MR3671992} is a generalization of the classical Schr\"{o}dinger equation, incorporating a fractional time derivative to model systems with anomalous diffusion and memory effects. Specifically, the time derivative is replaced by the Caputo derivative of order \(\alpha \in (0,1]\), allowing for a more flexible description of non-local dynamics in quantum systems. The equation for a continuous medium typically takes the form:
\begin{equation}\label{eq:time-frac-schroedinger}
i^\alpha \hbar_\alpha \psi^{(\alpha)}(x,t) = -\frac{\hbar^2}{2\hat{m}} \Delta \psi(x,t) + V(x,t) \psi(x,t),
\end{equation}
where \(\psi(x,t)\) is the wave function, \(V(x,t)\) is the potential, \(\hbar\) is the reduced Planck's constant, \(\hat{m}\) is the particle mass, and $\hbar_\alpha$ is a scaling coefficient with physical dimension \si{erg.sec^\alpha}; we refer back to~\cite{MR3671992} and~\cite{10.1063/1.1769611} for details on the interpretation. 

To formulate the FDE as a system—whether autonomous or not—we employ a semi-discretization of the spatial variables using the finite element method. In all subsequent examples, we consider a square domain $\Omega$ discretized with a triangular mesh, as illustrated in Fig.~\ref{fig:domain_and_mesh:domain}. We utilize quadratic Lagrangian elements, meaning that each triangular element has nodes at its vertices and edge midpoints. Dirichlet boundary conditions are enforced using either the \emph{stiff-spring} method, which penalizes the degrees of freedom associated with the boundary conditions for the time-{independent} case, or the subspace method for the time-dependent potential--{i.e.}, the corresponding degrees of freedom are removed from the equation. 
\begin{figure}[htbp]
    \centering
    \subfigure[\label{fig:domain_and_mesh:domain}]{\input{domainandmesh}}\hspace{0.022\columnwidth}
    \subfigure[\label{fig:domain_and_mesh:stiff}]{\includegraphics[width=0.247\columnwidth]{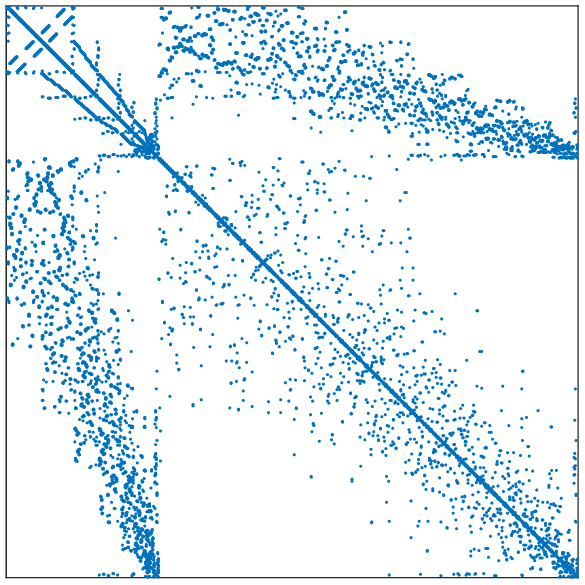}}
    
    \caption{Domain, mesh and pattern of the stiffness matrix for the FEM discretization of the~\eqref{eq:time-frac-schroedinger} equation on a square domain with a maximum size of the element of $h_{\max} = 0.3$.}
    \label{fig:domain_and_mesh}
\end{figure}
With this approach, we obtain a system of FDEs of the form  
\[
\begin{cases} 
\mathsf{M}_N \boldsymbol{\psi}^{(\alpha)}(t) = \mathsf{K}_N(t) \boldsymbol{\psi}(t), & t \in (0,T], \\ 
\boldsymbol{\psi}(0) = \boldsymbol{\psi}_0, 
\end{cases} 
\]  
which can then be solved using the procedure based on the solution formulation presented in~\eqref{eq:star-solution2}; The structure of the two mass matrices, $\mathsf{M}_N$ and $\mathsf{K}_N(t)$, is depicted in Fig.~\ref{fig:domain_and_mesh:stiff}. For the chosen problem formulation, the time-dependent terms in $K(t)$ appears as a single time-dependent function multiplying a matrix.

\subsubsection{Time-independent potential}\label{sec:time-indepedent}
As a first numerical experiment, we consider the test case from~\cite{MR3689930}, which involves a variation of~\eqref{eq:time-frac-schroedinger} defined on a two-dimensional square domain \(\Omega = [-2, 2] \times [-2, 2]\). The potential \(V(x)\) is set to zero within the interior of the subdomain \(\Omega_p = [-1, 1] \times [-1, 1]\) and is constant outside, specifically \(V(x) = 0\) for \(x \in \mathrm{int}(\Omega_p)\) and \(V(x) = 10\) for \(x \in \mathrm{cl}(\Omega / \Omega_p)\). The initial wavefunction at time \(t = 0\) is chosen to be of Gaussian type, given by \(\psi(0, x) = e^{-|x|^2 / 2}\). For simplicity, all constants are normalized as \(\hbar = \hbar_{\alpha} = \hat{m} = 1\).
\begin{figure}[htbp]
    \centering
    \input{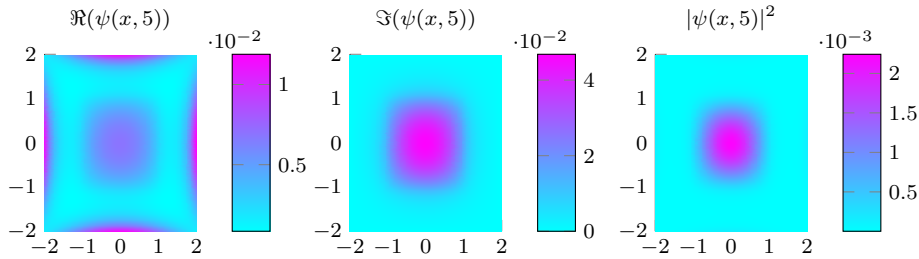}
    
    \caption{Solution of the test problem in Section~\ref{sec:time-indepedent} for $\alpha = 0.8$ and employing the Fractional Linear Multistep Method based on BDF2 formula with $\Delta t = 10^{-2}$.}
    \label{fig:schrodinger-fractional-time-independent}
\end{figure}
Figure~\ref{fig:schrodinger-fractional-time-independent} illustrates a solution to the problem for $\alpha = 0.8$ at $T = 5$.  

In Table~\ref{tab:star_time_indepedent}, we compare the solution obtained using the $\star$-approach, as described in Theorem~\ref{thm:star-solution}, with the solution computed via the FBDF2 method implemented in the \texttt{flmm} code~\cite{garrappa2018numerical,MR3327641}. 
\begin{table}[htbp]
    \centering
    \setlength{\tabcolsep}{4pt}
    \begin{tabular}{cccccc|cccc}
    \multicolumn{6}{c}{$\star$-approach} & \multicolumn{4}{c}{Fractional BDF2 (via \texttt{flmm2})} \\
    \toprule
    $m$ & $k$ & Abs.Err. & Rel.Err. & $\operatorname{rk}(\mathsf{X})$ & T (\si{s}) & $\Delta t$ & T (\si{s}) & Abs.Err. & Rel.Err. \\
    \midrule
 100 & 34 & 3.45e-03 & 9.17e-04 & 19 & 1.78 & $10^{-1}$ & 0.79 & 1.19e-02 & 3.16e-03 \\
 500 & 116 & 2.84e-04 & 7.55e-05 & 23 & 6.18 & $10^{-2}$ & 4.61 & 2.46e-04 & 6.54e-05  \\
 1000 & 196 & 9.94e-05 & 2.65e-05 & 24 & 17.27 & $10^{-3}$ & 44.41 & 4.07e-06 & 1.08e-06 \\
 1500 & 268 & 4.91e-05 & 1.31e-05 & 24 & 33.50 & $10^{-4}$ & 436.16 & 5.12e-08 & 1.36e-08  \\
 2000 & 332 & 2.85e-05 & 7.57e-06 & 24 & 60.30 \\
 \midrule
    \end{tabular}
    \caption{Solution of the discretized Schrodinger with time-independent potential and $\alpha = 0.5$. The maximum size of the mesh triangles is set to $0.3$. The reference solution, with respect to which both absolute and relative error are computed, is obtained with the Fractional BDF2 approach implemented in the \texttt{flmm2} routine and $\Delta t = 10^{-5}$.}
    \label{tab:star_time_indepedent}
\end{table}
To provide a reference solution, we employ the FBDF2 method with a fine time step $\Delta t = 10^{-5}$. To mitigate the computational cost and memory usage associated with FBDF2, we restrict the final time to $T = 1${,} as with any fractional linear multistep method, the entire integration history must be stored. For the $\star$-approach, we vary the number of Legendre expansion coefficients $m$ from 100 to 2000. In contrast, for the FBDF2 method, we decrease the time step~$\Delta t$ from~$10^{-1}$ to~$10^{-4}$. 

As in previous examples, we denote by $k$ the number of retained coefficients. Comparing the runtimes of the two methods, we observe that the $\star$-approach has a comparable running time for comparable accuracies, while aiming for higher accuracy with the FBDF2 approach is significantly more expensive. Additionally, it offers the advantage that the final solution is compactly represented by only~$k$ coefficients, rather than a dense matrix of size proportional to the number of spatial degrees of freedom and $\nicefrac{1}{\Delta t}$. In this case, nevertheless the time-stepping is still more efficient for this small case. When we refine the spatial mesh setting the maximum size of the mesh triangle to $0.1$, as reported in Table~\ref{tab:refined_case}, we observe a clear advantage with respect to the \texttt{flmm2} approach.
\begin{table}[htbp]
    \centering
    \setlength{\tabcolsep}{4pt}
    \begin{tabular}{cccccc|cccc}
    \multicolumn{6}{c}{$\star$-approach} & \multicolumn{4}{c}{Fractional BDF2 (via \texttt{flmm2})} \\
    \toprule
    $m$ & $k$ & Abs.Err. & Rel.Err. & $\operatorname{rk}(\mathsf{X})$ & T (\si{s}) & $\Delta t$ & T (\si{s}) & Abs.Err. & Rel.Err. \\
    \midrule
 100 & 34 & 3.05e-02 & 2.72e-03 & 13 & 27.39 & $10^{-1}$ & 101.37 & 3.69e-02 & 3.29e-03 \\ 
 500 & 116 & 6.03e-04 & 5.38e-05 & 21 & 141.23 & $10^{-2}$ & 559.08 & 7.62e-04 & 6.80e-05 \\
 1000 & 196 & 1.93e-04 & 1.72e-05 & 23 & 297.29 & $10^{-3}$ & 5108.31 & 1.26e-05 & 1.12e-06 \\
 1500 & 266 & 1.02e-04 & 9.07e-06 & 23 & 482.68 & $10^{-4}$ & 50651.23 & 1.59e-07 & 1.42e-08 \\
 2000 & 330 & 6.50e-05 & 5.80e-06 & 24 & 666.72  \\    
    \midrule
    \end{tabular}
    \caption{Solution of the discretized Schrodinger with time-independent potential and $\alpha = 0.5$. The maximum size of the mesh triangles is set to $0.1$. The reference solution, with respect to which both absolute and relative error are computed, is obtained with the Fractional BDF2 approach implemented in the \texttt{flmm2} routine and $\Delta t = 10^{-5}$.}
    \label{tab:refined_case}
\end{table}
This is due both to the increased cost in the solution of the linear systems inside the time-stepping procedure and to the increase in memory consumption for the sum of the integration queue. For both levels of refinement presented in Tables~\ref{tab:star_time_indepedent} and~\ref{tab:refined_case}, corresponding to the finite element discretization, we observe that the rank of the matrix $\mathsf{X}$ in equation~\eqref{eq:matrix_equation_complete} is substantially lower than the number of coefficients $m$ in the chosen basis expansion. This observation indicates that the solution of the reduced system~\eqref{eq:matrix_eq_projected} can be obtained to the desired tolerance with a relatively small value of $j$, which corresponds to the rank of $\mathsf{X}$. Consequently, the procedure is computationally efficient from this perspective.

\subsubsection{Time-dependent potential}\label{sec:time-dependent}

For this case we consider the problem~\eqref{eq:time-frac-schroedinger} with a potential $V(x,t)$ that actually depends on time in order to use the representation of the solution described in Corollary~\ref{cor:nonautonomous_system}. Specifically, we consider the same spatial domain of Section~\ref{sec:time-indepedent} and modify the potential {as} %
\[
V_2(x,t) =  V_1(x,y) + \frac{1}{2}\left( 1 + 0.1 \sin(5 \pi^2 t) \right),
\]
for $V(x)$ the potential which is set to zero within the interior of the subdomain \(\Omega_p = [-1, 1] \times [-1, 1]\) and is constant outside, specifically \(V(x) = 0\) for \(x \in \mathrm{int}(\Omega_p)\) and \(V(x) = 10\) for \(x \in \mathrm{cl}(\Omega / \Omega_p)\).
\begin{table}[htbp]
    \centering
    \setlength{\tabcolsep}{4pt}
    \begin{tabular}{cccccc|cccc}
    \multicolumn{6}{c}{$\star$-approach} & \multicolumn{4}{c}{Fractional BDF2 (via \texttt{flmm2})} \\
    \toprule
    $m$ & $k$ & Abs.Err. & Rel.Err. & $\operatorname{rk}(\mathsf{X})$ & T (\si{s}) & $\Delta t$ & T (\si{s}) & Abs.Err. & Rel.Err. \\
    \midrule
100	&	30	&	6.98e-03	&	2.08e-03	&	16	&	2.53 &  $10^{-1}$& 2.81 & 1.12e-02 & 3.33e-03 \\
500	&	92	&	1.54e-03	&	4.58e-04	&	18	&	3.73 & $10^{-2}$ & 11.70 & 6.78e-04 & 2.02e-04 \\
1000	&	154	&	7.51e-04	&	2.24e-04	&	19	&	5.59 &  $10^{-3}$& 103.52 & 2.29e-05 & 6.81e-06 \\
2000	&	256	&	3.63e-04	&	1.08e-04	&	19	&	22.58 &  $10^{-4}$ & 1014.80 & 9.13e-07 & 2.72e-07 \\
2500	&	302	&	2.88e-04	&	8.57e-05	&	19	&	36.00 \\
3000 & 346 & 2.38e-04 & 7.10e-05 & 19 & 56.12 \\
3500 & 388 & 2.04e-04 & 6.06e-05 & 19 & 75.18 \\
    \bottomrule
    \end{tabular}
    \caption{Solution of the discretized Schrodinger with time-dependent potential and $\alpha = 0.3$. The maximum size of the mesh triangles is set to $0.3$. The reference solution, with respect to which both absolute and relative error are computed, is obtained with the Fractional BDF2 approach implemented in the \texttt{flmm2} routine and $\Delta t = 10^{-5}$.}
    \label{tab:time_dependent_unrefined}
\end{table}
The effectiveness of the $\star$‑approach in handling time-dependent potentials is further examined on both coarse and fine spatial meshes, with results summarized in Tables~\ref{tab:time_dependent_unrefined} and~\ref{tab:time_dependent_refined_case}. The fractional order is fixed at $\alpha = 0.3$ in all cases. On the coarser mesh with a maximum triangle size of $h_{\max} = 0.3$, the $\star$‑approach achieves relative errors decreasing from $2.08 \times 10^{-3}$ to $8.57 \times 10^{-5}$ as the number of basis functions $m$ increases from 100 to 2500. The rank of the solution matrix $\mathsf{X}$ remains stable, growing slowly from 16 to 19, while computational time increases moderately from approximately 2.5 seconds to 36 seconds. In contrast, the FBDF2 method implemented via the \texttt{flmm2} routine requires over 1000 seconds to reach a relative error of $2.72 \times 10^{-7}$ for $\Delta t = 10^{-4}$, showing that comparable or superior accuracy can be obtained with significantly lower computational effort using the $\star$‑approach. For example, with $m = 500$, the $\star$‑approach attains a relative error of $4.58 \times 10^{-4}$ in only 3.73 seconds, while FBDF2 with $\Delta t = 10^{-3}$ needs 11.7 seconds to achieve $2.04 \times 10^{-4}$.
\begin{table}[htbp]
    \centering
    \setlength{\tabcolsep}{4pt}
    \begin{tabular}{cccccc|cccc}
    \multicolumn{6}{c}{$\star$-approach} & \multicolumn{4}{c}{Fractional BDF2 (via \texttt{flmm2})} \\
    \toprule
    $m$ & $k$ & Abs.Err. & Rel.Err. & $\operatorname{rk}(\mathsf{X})$ & T (\si{s}) & $\Delta t$ & T (\si{s}) & Abs.Err. & Rel.Err. \\
    \midrule
100	&	30	&	2.16e-02	&	2.20e-03	&	13	&	61.75 &  $10^{-1}$ & 502.99 & 3.31e-02 & 3.37e-03 \\
500	&	92	&	4.80e-03	&	4.88e-04	&	15	&	65.92 &  $10^{-2}$ & 1358.86 & 2.04e-03 & 2.08e-04 \\
1000	&	154	&	2.42e-03	&	2.46e-04	&	16	&	67.84  &  $10^{-3}$ & 9713.20 & 7.46e-05 & 7.59e-06 \\
1500	&	206	&	1.63e-03	&	1.66e-04	&	16	&	64.54 &  $10^{-4}$ & 94326.94 & 3.44e-06 & 3.50e-07 \\
2000 & 256 & 1.25e-03 & 1.27e-04  & 17 & 61.48 \\
2500 & 300 & 1.02e-03 & 1.04e-04  & 17 & 75.42 \\
3000 & 344 & 8.70e-04 & 8.85e-05  & 17 & 116.13 \\
3500 & 384 & 7.64e-04 & 7.77e-05  & 18 & 133.50 \\
    \midrule
    \end{tabular}
    \caption{Solution of the discretized Schrodinger with time-dependent potential and $\alpha = 0.3$. The maximum size of the mesh triangles is set to $0.1$. The reference solution, with respect to which both absolute and relative error are computed, is obtained with the Fractional BDF2 approach implemented in the \texttt{flmm2} routine and $\Delta t = 10^{-5}$.}
    \label{tab:time_dependent_refined_case}
\end{table}
On the finer mesh with $h_{\max} = 0.1$, the performance trends persist. The $\star$‑approach produces relative errors ranging from $2.20 \times 10^{-3}$ to $1.04 \times 10^{-4}$ as $m$ increases from 100 to 2500, with the numerical rank of $\mathsf{X}$ remaining around 13–16 for $m \leq 1500$. At $m = 2000$, the $\star$‑approach achieves a relative error of $1.27 \times 10^{-4}$ within approximately 62 seconds, whereas FBDF2 at a moderate comparable accuracy of $4.88 \times 10^{-4}$ requires nearly 23 minutes with $\Delta t = 10^{-2}$, and a similar behavior is also observed in the time ratio when the error is of the order of $10^{-5}$. These results emphasize the efficiency and scalability of the proposed method. The computational advantage becomes increasingly pronounced as the spatial mesh is refined. Moreover, the ability to maintain a low-rank representation throughout the simulation demonstrates the method’s effectiveness in capturing the essential dynamics of time-fractional Schr\"{o}dinger equations with non-autonomous potentials.

\section{Conclusions}\label{sec:conclusions}
\setcounter{section}{5} \setcounter{equation}{0} %

In this work{,} we have developed a novel algebraic framework based on the $\star$‐product for the solution of linear nonautonomous fractional differential equations FDEs. By reformulating the Caputo FDE in terms of two‐variable distributions, we derived a compact representation of the solution as a $\star$‐resolvent. This reformulation together with the path‐sum combinatorics technique allowed us to express as series in generalized hypergeometric and Mittag–Leffler functions {as} the solution of some systems of nonautonomous FDEs.

To render the approach computationally viable, we proposed a ``solve-then-discretize'' strategy: all distributions are expanded in a truncated Legendre basis and $\star$‐operations are mapped to finite matrix algebra, reducing the problem to a single linear system. Numerical experiments on scalar autonomous and nonautonomous examples, as well as on the time‐fractional Schr\"{o}dinger equation, demonstrate that the $\star$‐method attains competitive accuracy and efficiency compared to classical fractional linear multistep methods, while offering a sparse spectral representation of the solution; in particular, for moderate spatial discretizations the $\star$‐approach outperforms FBDF2 in both runtime and memory usage, and its Krylov–Schur variant scales favorably for large systems.

Several avenues for future research emerge naturally: a rigorous error analysis of the Legendre truncation and adaptive strategies for selecting the basis size and cutoff could enhance robustness; alternative non‐polynomial bases, such as fractional or multiscale wavelets, may better capture solution regularity near $t=0$; the development of tailored preconditioners and low‐rank iterative methods for the block‐structured matrix equations would enable large‐scale space–time fractional PDE applications.

\begin{acknowledgements}
 F.D. is a member of the INdAM GNCS group, and his research was partially granted by the 
 Italian Ministry of University and Research (MUR) through the PRIN 2022 ``MOLE: Manifold 
 constrained Optimization and LEarning'',  code: 2022ZK5ME7 MUR D.D. financing decree 
 n. 20428 of November 6th, 2024 (CUP B53C24006410006).  P.-L.G. and S.P. work was funded by the French National Research Agency project \textsc{Magica} ANR-20-CE29-0007. 
 S.P. work was also supported by Charles University Research programs PRIMUS/21/SCI/009 and UNCE24/SCI/005.
 F.D. acknowledges the MUR Excellence Department Project awarded to the Department of Mathematics, University of Pisa, CUP I57G22000700001.

 The authors are grateful to the two anonymous referees for their thorough reading of the manuscript and for their helpful comments. Their suggestions have contributed to improving the clarity and presentation of the paper.
 \end{acknowledgements}

 \section*{\small
 Conflict of interest} %
 {\small
 The authors declare that they have no conflict of interest.}

\setcounter{section}{0}
\renewcommand*{\thesection}{\Alph{section}.}
\section{Appendix: path-sum expressions for system \ref{SysExample}}\label{AppA}
We have shown in the main text that $\mathsf{G}_{11}=\big(1_\star - \Theta^{\star \alpha})^{\star -1}\star (1_\star - t\Theta^{\star \alpha})^{\star -1}$. Given the results of Examples~\ref{example:mlfunction} and \ref{example:nonautonomous} we have
\begin{align*}
(1_\star - t\Theta^{\star \alpha})^{\star -1}\star \Theta^{\star(1-\alpha)}&=\frac{1}{\Gamma
   \left(\frac{1}{\alpha +1}\right)}\sum_{k=0}^\infty \frac{\Gamma \left(k-\frac{\alpha }{\alpha +1}
   \right) }{ \Gamma (k+k\alpha- \alpha )}(\alpha +1)^{k-1}\,t^{\alpha  (k-1)+k},\\
\Theta^{\star \alpha}\star (1_\star - \Theta^{\star \alpha})^{\star -1}&=(t-s)^{\alpha -1} E_{\alpha ,\alpha }\big((t-s)^{\alpha
   }\big),
\end{align*}
where $E_{\alpha ,\beta }(x)$ {is the generalized Mittag-Leffler function~\eqref{eq:mittag-leffler-definition}.} %
Recall that $\mathsf{U}_{11}=\Theta^{\star \alpha}\star \mathsf{G}_{11}\star \Theta^{\star(1-\alpha)}$. Combining this with the above leads to the expression for $\mathsf{U}_{11}$ given in the main text.

We proceed similarly for $\mathsf{U}_{22}=\Theta^{\star \alpha}\star \mathsf{G}_{22}\star \Theta^{\star(1-\alpha)}$, starting with the path-formulation of $\mathsf{G}_{22}$ in terms of simple cycles on the graph $\mathcal{G}$,
\begin{align*}
\mathsf{G}_{22}&=\Big(1_\star  -\underbrace{(1\Theta^{\star \alpha})\star (1_\star - \overbrace{(1+t)\Theta^{\star \alpha}}^{\text{Loop } 1\leftarrow 1}\,)^{\star-1}\star(-t\Theta^{\star \alpha}}_{\text{Backtrack }2\leftarrow 1\leftarrow 2})\Big)^{\star -1}.
\end{align*}
Evaluating now the $\star$-Neumann series for the inner and outer resolvents leads~to
$$
\mathsf{U}_{22}=1-\frac{1}{\Gamma
   \left(\frac{1}{\alpha +1}\right)}\sum_{k=1}^\infty\sum_{m=0}^\infty \frac{(\alpha +1)^k\, \Gamma \left(k+\frac{1}{\alpha +1}\right) }{ \Gamma (k+(k+m+1) \alpha +1)}t^{\alpha  (k+m+1)+k}.
$$
The path-sum formulation for off-diagonal terms involves the simple paths of $\mathcal{G}$, for example
$$
\mathsf{G}_{21}=\overbrace{\underbrace{(1\Theta^{\star \alpha})}_{\text{Edge } 2\leftarrow 1}\,\,\,\star \underbrace{\,\mathsf{G}_{11}\,}_{\text{Cycles }1\leftarrow 1}}^{\text{Path } 2\leftarrow 1}\,,
$$
and
$$
\mathsf{G}_{12}=\underbrace{\overbrace{\big(1_\star -(1+t)\Theta^{\star \alpha}\big)^{\star -1}}^{\text{Cycles } 1\leftarrow 1 \text{ on } \mathcal{G}\backslash\{2\}}\star \overbrace{(1\Theta^{\star \alpha})}^{\text{Edge }1\leftarrow 2}\,\,\star\! \!\overbrace{\,\mathsf{G}_{22}\,}^{\text{Cycles } 2\leftarrow 2}}_{\text{Path } 1\leftarrow 2}.
$$
Calculations of the $\star$-products and resolvents finally give
\begin{align*}
\mathsf{U}_{21}&=\frac{1}{\Gamma \left(\frac{1}{\alpha
   +1}\right)}\sum_{k,m=0}^\infty \frac{(\alpha +1)^k\, \Gamma \left(k+\frac{1}{\alpha +1}\right)
   }{ \Gamma (k+(k+m+1) \alpha +1)}t^{\alpha  (k+m+1)+k},\\
\mathsf{U}_{12}&=-\frac{1}{\Gamma
   \left(\frac{1}{\alpha +1}\right)}\sum_{k=1}^\infty\sum_{m=0}^\infty\frac{(\alpha +1)^k\, \Gamma \left(k+\frac{1}{\alpha +1}\right) }{ \Gamma (k+(k+m) \alpha +1)}t^{\alpha  (k+m)+k}.
\end{align*}

\section{Coefficient matrix computation}\label{app:coeff}
We aim at computing $\mathsf{F}_m$, the coefficient matrix \eqref{eq:coeff:mtx} of the distribution $f(t,s):=\tilde{f}(t)(\Theta(t-s))^{\star \alpha}$.
Note that
\[ f(t,s) = \tilde{f}(t)\delta(t-s) \star \Theta(t-s)^{\star \alpha}, \]
At the end of Section~\ref{sec:discretization}, we explained how to compute $\mathsf{H}_m^\alpha$, the coefficient matrix of $\Theta(t-s)^{\star \alpha}$.
Hence, if we can compute the coefficient matrix $\mathsf{F}_\delta$ of $\tilde{f}(t)\delta(t-s)$, we can approximate $\mathsf{F}_m$ by matrix multiplication:
\[ \mathsf{F}_m \approx \mathsf{F}_\delta  \mathsf{H}_m^\alpha.  \]
Let us first expand $\tilde{f}(t)$ into the Legendre polynomial basis, i.e., $$\tilde{f}(t) =  \sum_{j=0}^\infty \beta_j P_j(t).$$
Then the $\mathsf{F}_\delta$ elements $f_{k,\ell}$ are given by
\begin{align*}
   f_{k,\ell} &= \iint_{I\times I} f(\tau,\sigma) P_k(\tau) P_\ell(\sigma) \; \mathrm{d} \tau \; \mathrm{d} \sigma \\
   &= \iint_{I\times I} \tilde{f}(\tau)\delta(\tau-\sigma) P_k(\tau) P_\ell(\sigma) \; \mathrm{d} \tau \; \mathrm{d} \sigma \\
   &= \int_{I} \tilde{f}(\sigma) P_k(\sigma) P_\ell(\sigma) \; \mathrm{d} \sigma = \sum_{j=0}^\infty \beta_j \underbrace{\int_{I} P_j(\sigma) P_k(\sigma) P_\ell(\sigma) \; \mathrm{d} \sigma}_{\mathcal{F}_{jk\ell}}\\
   &= \sum_{j=0}^\infty \beta_j \mathcal{F}_{jk\ell}.
\end{align*}
Without loss of generality, assume that the interval is $I=[-1,1]$. 
Then $\mathcal{F}_{jk\ell}$ can be computed by the method described in \cite{pozza2023new}, while the expansion of $\tilde{f}(t)$ can be truncated and the related coefficients $\beta_j$ can be approximated by the MATLAB package \verb|chebfun|  \cite{chebfun}.

\section{Iterative algorithm for the matrix-equation \eqref{eq:time_dependent_mex}}
\label{sec:solution_alg_details}

We seek a low-rank approximation of the solution of the multi-term matrix equation~\eqref{eq:time_dependent_mex}
\[
  \mathsf{X}
  \;-\;
  \mathsf{H}_m^\alpha\,\mathsf{X}\,\mathsf{K}^\top
  \;-\;
  \mathsf{F}_m\,\mathsf{X}\,\mathsf{L}^\top
  \;=\;
  \phi_m^{(\alpha,0)}\,u_0^\top.
\]
To obtain a simple solution procedure, we rewrite it in the fixed‐point form:
\begin{equation*}
  \mathsf{X}_{k+1}
  \;-\;
  \mathsf{H}_m^\alpha\,\mathsf{X}_{k+1}\,\mathsf{K}^\top
  \;=\;
  \mathsf{F}_m\,\mathsf{X}_{k}\,\mathsf{L}^\top
  \;+\;
  \phi_m^{(\alpha,0)}\,u_0^\top, \;\;\; k = 0, 1, \dots, \;\;\; \mathsf{X}_0 = 0.
\end{equation*}
In our tests, we have observed numerically that the iterates $\mathsf{X}_k$ can be approximated by a low-rank matrix: $\mathsf{X}_k \approx \mathsf{L}_k \mathsf{R}_k^\top$. We denote with $s(k)$, the number of columns of $\mathsf{L}_k$ and $\mathsf{R}_k$.  
Then, we can approximate $\mathsf{X}_{k+1}$ by solving the matrix equation
\begin{align}\label{eq:mtx:nonaut:it}
  \mathsf{X}_{k+1}
  -
  \mathsf{H}_m^\alpha\,\mathsf{X}_{k+1}\,\mathsf{K}^\top
  =
  \mathsf{B}_k \mathsf{C}_k^\top, 
  \;\;
  \mathsf{B}_k = [\mathsf{F}_m\,\mathsf{L}_{k}, \, \phi_m^{(\alpha,0)}], 
  \;\;
  \mathsf{C}_k := [\mathsf{L}\, \mathsf{R}_k, \, u_0].
\end{align}
In practice, we approximate $\mathsf{B}_k$ and $\mathsf{C}_k$ using lower-rank matrices $\hat{\mathsf{B}}_k$ and $\hat{\mathsf{C}}_k$.
Since this latter matrix equation has two terms and a low-rank right-hand side, we can now rely on a Krylov subspace approach for its solution; see, e.g., \cite{Sim16}. Let us consider the block Krylov subspace
\[ \mathcal{K}_q (\mathsf{K}, \hat{\mathsf{C}}_k) : = \text{span} \{\mathsf{K}\,\hat{\mathsf{C}}_k, \mathsf{K}^2\,\hat{\mathsf{C}}_k, \dots, \mathsf{K}^{q-1}\,\hat{\mathsf{C}}_k \}, \]
and assume that its dimension is maximal, i.e., $q \cdot s(k)$ (this has proven to be true in our numerical tests). By running the \emph{block Arnoldi algorithm} with reorthogonalization (e.g., \cite{saad03}) we obtain the orthogonal matrix $\mathsf{V}_q$, basis of $\mathcal{K}_q (\mathsf{K}, \hat{\mathsf{C}}_k)$. We then approximate the solution of \eqref{eq:mtx:nonaut:it} as $\mathsf{X}^\top \approx \mathsf{V}_q \mathsf{Y}$. Consequently, we get the equation 
\[
     \mathsf{V}_q \mathsf{Y}
  \;-\;
  \mathsf{K} \,\mathsf{V}_q \mathsf{Y}\, (\mathsf{H}_m^\alpha)^\top
  \;=\;
  \hat{\mathsf{C}}_k \hat{\mathsf{B}}_k^\top,
\]
and multiplying by $\mathsf{V}_q^H$ we get the reduced order equation
\[
     \mathsf{Y}
  \;-\;
  \mathsf{V}_q^H \mathsf{K} \,\mathsf{V}_q \mathsf{Y}\, (\mathsf{H}_m^\alpha)^\top
  \;=\;
  (\mathsf{V}_q^H \hat{\mathsf{C}}_k) \hat{\mathsf{B}}_k^\top.
\]
Note that the matrix $\mathsf{J}_q = \mathsf{V}_q^H \mathsf{K} \,\mathsf{V}_q$ has a much smaller size than $\mathsf{K}$ and is an output of the Arnoldi algorithm. Now that the equation has a smaller size, it can be solved by a standard explicit method, e.g., by the Schur decomposition approach (e.g., \cite{Sim16}). We remark that the Schur decomposition of $\mathsf{H}_m$ can be computed once for all and stored.
As a stopping criterion for the iteration, we consider the difference of two consecutive iterates $\mathsf{L}_{k+1} \mathsf{R}_{k+1}^\top - \mathsf{L}_{k} \mathsf{R}_{k}k^\top$ and we multiply it from the left by $|\phi_m^{(\alpha,0)}|^\top$ (the absolute value is applied componentwise) and from the right by $u_0$ obtaining the residual norm estimates
\[\left| (|\phi_m^{(\alpha,0)}|^\top \mathsf{L}_{k+1})(\mathsf{R}_{k+1}^\top u_0) - (|\phi_m^{(\alpha,0)}|^\top \mathsf{L}_{k})(\mathsf{R}_{k}^\top u_0)  \right|{,}\]
which is much cheaper than computing the norm of the difference of the two consecutive iterates. The choice of the two vectors is a heuristic based on the fact that $u_0^\top \tilde{u}(T) \approx |\phi_m(0)^\top| \mathsf{H}_m^\alpha \mathsf{X} u_0$. 
Overall, we get Algorithm \ref{algo:it:arn}.
\begin{algo}\label{algo:it:arn}
{Low-Rank Iterative Scheme with Block Arnoldi}
\begin{algorithmic}[1]

\State \textbf{Inputs:}
\State \quad $\mathsf{F}_m$, $\mathsf{H}_m^\alpha$, $\mathsf{L}$, $\mathsf{K}$, $\phi_m^{(\alpha,0)}$, $u_0$ \Comment{Matrix equation data}
\State \quad $tol$ \Comment{Tolerance for stopping criterion}

\State \textbf{Outputs:}
\State \quad $\mathsf{L}_k$, $\mathsf{R}_k$ \Comment{Factors of the approximated solution}

\Statex

\State \textbf{Initialize:}
\State \quad $\mathsf{L}_0 \gets 0$, \quad $\mathsf{R}_0 \gets 0$, \quad $k \gets 0$

\While{$res > tol$}
    \State Construct block matrices:
    \State \quad $\mathsf{B}_k \gets [\mathsf{F}_m \, \mathsf{L}_k,\; \phi_m^{(\alpha,0)}]$
    \State \quad $\mathsf{C}_k \gets [\mathsf{L} \, \mathsf{R}_k,\; u_0]$
    
    \State Apply low-rank approximation:
    \State \quad $\hat{\mathsf{B}}_k \approx \mathsf{B}_k$
    \State \quad $\hat{\mathsf{C}}_k \approx \mathsf{C}_k$
    
    \State Perform Block Arnoldi:
    \State \quad $(\mathsf{V}_q,\, \mathsf{J}_q) \gets \text{BlockArnoldi}(\mathsf{K},\, \hat{\mathsf{C}}_k)$
    
    \State Solve Stein matrix equation:
    \State \quad $\mathsf{Y} - \mathsf{J}_q \mathsf{Y} (\mathsf{H}_m^\alpha)^\top = (\mathsf{V}_q^H \hat{\mathsf{C}}_k) \hat{\mathsf{B}}_k^\top$
    
    \State Update low-rank factors:
    \State \quad $\mathsf{L}_{k+1} \gets \mathsf{Y}^\top$
    \State \quad $\mathsf{R}_{k+1} \gets \mathsf{V}_q$
    
    \State Residual norm approximation:
    \State \quad $res \gets | (|\phi_m^{(\alpha,0)}|^\top \mathsf{L}_{k+1})(\mathsf{R}_{k+1}^\top u_0) - (|\phi_m^{(\alpha,0)}|^\top \mathsf{L}_{k})(\mathsf{R}_{k}^\top u_0)  | $
    \State $k \gets k + 1$
\EndWhile

\end{algorithmic}
\end{algo}

\bibliographystyle{spmpsci}
\bibliography{bibliografia}

\begin{thebibliography}{10}
\providecommand{\url}[1]{{#1}}
\providecommand{\urlprefix}{URL }
\expandafter\ifx\csname urlstyle\endcsname\relax
  \providecommand{\doi}[1]{DOI~\discretionary{}{}{}#1}\else
  \providecommand{\doi}{DOI~\discretionary{}{}{}\begingroup
  \urlstyle{rm}\Url}\fi

\bibitem{MR4410753}
Aceto, L., Novati, P.: Fast and accurate approximations to fractional powers of
  operators.
\newblock IMA J. Numer. Anal. \textbf{42}(2), 1598--1622 (2022).
\newblock \doi{10.1093/imanum/drab002}.
\newblock \urlprefix\url{https://doi.org/10.1093/imanum/drab002}

\bibitem{BagleyTorvik}
Bagley, R.L., Torvik, P.J.: {A Theoretical Basis for the Application of
  Fractional Calculus to Viscoelasticity}.
\newblock Journal of Rheology \textbf{27}(3), 201--210 (1983).
\newblock \doi{10.1122/1.549724}.
\newblock \urlprefix\url{https://doi.org/10.1122/1.549724}

\bibitem{MR3498143}
Breiten, T., Simoncini, V., Stoll, M.: Low-rank solvers for fractional
  differential equations.
\newblock Electron. Trans. Numer. Anal. \textbf{45}, 107--132 (2016)

\bibitem{CaputoDerivative}
Caputo, M.: Linear models of dissipation whose q is almost frequency
  independent—ii.
\newblock Geophysical Journal International \textbf{13}(5), 529--539 (1967).
\newblock \doi{10.1111/j.1365-246X.1967.tb02303.x}.
\newblock \urlprefix\url{https://doi.org/10.1111/j.1365-246X.1967.tb02303.x}

\bibitem{MR2386503}
Chen, J., Liu, F., Anh, V.: Analytical solution for the time-fractional
  telegraph equation by the method of separating variables.
\newblock J. Math. Anal. Appl. \textbf{338}(2), 1364--1377 (2008).
\newblock \doi{10.1016/j.jmaa.2007.06.023}.
\newblock \urlprefix\url{https://doi.org/10.1016/j.jmaa.2007.06.023}

\bibitem{MR4379629}
Colbrook, M.J.: Computing semigroups with error control.
\newblock SIAM J. Numer. Anal. \textbf{60}(1), 396--422 (2022).
\newblock \doi{10.1137/21M1398616}.
\newblock \urlprefix\url{https://doi.org/10.1137/21M1398616}

\bibitem{DiethelmBook}
Diethelm, K.: The analysis of fractional differential equations, \emph{Lecture
  Notes in Mathematics}, vol. 2004.
\newblock Springer-Verlag, Berlin (2010).
\newblock \doi{10.1007/978-3-642-14574-2}.
\newblock \urlprefix\url{https://doi.org/10.1007/978-3-642-14574-2}.
\newblock An application-oriented exposition using differential operators of
  Caputo type

\bibitem{DIETHELM2006482}
Diethelm, K., Ford, J.M., Ford, N.J., Weilbeer, M.: Pitfalls in fast numerical
  solvers for fractional differential equations.
\newblock Journal of Computational and Applied Mathematics \textbf{186}(2),
  482--503 (2006).
\newblock \doi{https://doi.org/10.1016/j.cam.2005.03.023}.
\newblock
  \urlprefix\url{https://www.sciencedirect.com/science/article/pii/S0377042705001287}

\bibitem{MR1926466}
Diethelm, K., Ford, N.J., Freed, A.D.: A predictor-corrector approach for the
  numerical solution of fractional differential equations.
\newblock Nonlinear Dynam. \textbf{29}(1-4), 3--22 (2002).
\newblock \doi{10.1023/A:1016592219341}.
\newblock \urlprefix\url{https://doi.org/10.1023/A:1016592219341}.
\newblock Fractional order calculus and its applications

\bibitem{NIST:DLMF}
{\it NIST Digital Library of Mathematical Functions}.
\newblock http://dlmf.nist.gov/, Release 1.1.7 of 2022-10-15.
\newblock \urlprefix\url{http://dlmf.nist.gov/}.
\newblock F.~W.~J. Olver, A.~B. {Olde Daalhuis}, D.~W. Lozier, B.~I. Schneider,
  R.~F. Boisvert, C.~W. Clark, B.~R. Miller, B.~V. Saunders, H.~S. Cohl, and
  M.~A. McClain, eds.

\bibitem{chebfun}
Driscoll, T.A., Hale, N., Trefethen, L.N.: Chebfun Guide.
\newblock Pafnuty Publications, Oxford (2014)

\bibitem{MR3033699}
Ford, N.J., Rodrigues, M.M., Vieira, N.: A numerical method for the fractional
  {S}chr\"{o}dinger type equation of spatial dimension two.
\newblock Fract. Calc. Appl. Anal. \textbf{16}(2), 454--468 (2013).
\newblock \doi{10.2478/s13540-013-0028-5}.
\newblock \urlprefix\url{https://doi.org/10.2478/s13540-013-0028-5}

\bibitem{MR3350038}
Garrappa, R.: Numerical evaluation of two and three parameter
  {M}ittag-{L}effler functions.
\newblock SIAM J. Numer. Anal. \textbf{53}(3), 1350--1369 (2015).
\newblock \doi{10.1137/140971191}.
\newblock \urlprefix\url{https://doi.org/10.1137/140971191}

\bibitem{MR3327641}
Garrappa, R.: Trapezoidal methods for fractional differential equations:
  theoretical and computational aspects.
\newblock Math. Comput. Simulation \textbf{110}, 96--112 (2015).
\newblock \doi{10.1016/j.matcom.2013.09.012}.
\newblock \urlprefix\url{https://doi.org/10.1016/j.matcom.2013.09.012}

\bibitem{garrappa2018numerical}
Garrappa, R.: Numerical solution of fractional differential equations: A survey
  and a software tutorial.
\newblock Mathematics \textbf{6}(2), 16 (2018)

\bibitem{MR3689930}
Garrappa, R., Moret, I., Popolizio, M.: On the time-fractional {S}chr\"odinger
  equation: theoretical analysis and numerical solution by matrix
  {M}ittag-{L}effler functions.
\newblock Comput. Math. Appl. \textbf{74}(5), 977--992 (2017).
\newblock \doi{10.1016/j.camwa.2016.11.028}.
\newblock \urlprefix\url{https://doi.org/10.1016/j.camwa.2016.11.028}

\bibitem{MR3068601}
Garrappa, R., Popolizio, M.: Evaluation of generalized {M}ittag-{L}effler
  functions on the real line.
\newblock Adv. Comput. Math. \textbf{39}(1), 205--225 (2013).
\newblock \doi{10.1007/s10444-012-9274-z}.
\newblock \urlprefix\url{https://doi.org/10.1007/s10444-012-9274-z}

\bibitem{Gelfand1967}
Gel'fand, I.M., Shilov, G.E.: Generalized functions, volume 1: Properties and
  operations.
\newblock American Mathematical Monthly \textbf{377}, 1026 (1967).
\newblock \urlprefix\url{https://api.semanticscholar.org/CorpusID:124080985}

\bibitem{MR3945243}
Gilles, M.A., Townsend, A.: Continuous analogues of {K}rylov subspace methods
  for differential operators.
\newblock SIAM J. Numer. Anal. \textbf{57}(2), 899--924 (2019).
\newblock \doi{10.1137/18M1177810}.
\newblock \urlprefix\url{https://doi.org/10.1137/18M1177810}

\bibitem{Giscard2015}
Giscard, P.L., Lui, K., Thwaite, S.J., Jaksch, D.: An exact formulation of the
  time-ordered exponential using path-sums.
\newblock Journal of Mathematical Physics \textbf{56}(5), 053503 (2015).
\newblock \doi{10.1063/1.4920925}.
\newblock \urlprefix\url{https://doi.org/10.1063/1.4920925}

\bibitem{MR4191370}
Giscard, P.L., Pozza, S.: Lanczos-like algorithm for the time-ordered
  exponential: the {$*$}-inverse problem.
\newblock Appl. Math. \textbf{65}(6), 807--827 (2020).
\newblock \doi{10.21136/AM.2020.0342-19}.
\newblock \urlprefix\url{https://doi.org/10.21136/AM.2020.0342-19}

\bibitem{Giscard2013}
Giscard, P.L., Thwaite, S.J., Jaksch, D.: Evaluating matrix functions by
  resummations on graphs: the method of path-sums.
\newblock SIAM J. Matrix Anal. Appl. \textbf{34}(2), 445--469 (2013).
\newblock \doi{10.1137/120862880}.
\newblock \urlprefix\url{https://doi.org/10.1137/120862880}

\bibitem{Giscard2012}
Giscard, P.L., Thwaite, S.J., Jaksch, D.: Walk-sums, continued fractions and
  unique factorisation on digraphs (2015).
\newblock \urlprefix\url{https://arxiv.org/abs/1202.5523}

\bibitem{Hig08}
Higham, N.J.: {Functions of Matrices. {T}heory and Computation}.
\newblock SIAM, Philadelphia (2008)

\bibitem{HigLin13}
Higham, N.J., Lin, L.: An improved {S}chur--{P}adé algorithm for fractional
  powers of a matrix and their {F}réchet derivatives.
\newblock SIAM Journal on Matrix Analysis and Applications \textbf{34}(3),
  1341--1360 (2013).
\newblock \doi{10.1137/130906118}.
\newblock \urlprefix\url{https://doi.org/10.1137/130906118}

\bibitem{KingDrazin}
King, C.F.: A note on {D}razin inverses.
\newblock Pacific J. Math. \textbf{73}(2), 383--390 (1977).
\newblock \urlprefix\url{http://dml.mathdoc.fr/item/1102811925}

\bibitem{PhysRevE.62.3135}
Laskin, N.: Fractional quantum mechanics.
\newblock Phys. Rev. E \textbf{62}, 3135--3145 (2000).
\newblock \doi{10.1103/PhysRevE.62.3135}.
\newblock \urlprefix\url{https://link.aps.org/doi/10.1103/PhysRevE.62.3135}

\bibitem{MR1755089}
Laskin, N.: Fractional quantum mechanics and {L}\'evy path integrals.
\newblock Phys. Lett. A \textbf{268}(4-6), 298--305 (2000).
\newblock \doi{10.1016/S0375-9601(00)00201-2}.
\newblock \urlprefix\url{https://doi.org/10.1016/S0375-9601(00)00201-2}

\bibitem{PhysRevE.66.056108}
Laskin, N.: Fractional {S}chr\"odinger equation.
\newblock Phys. Rev. E \textbf{66}, 056108 (2002).
\newblock \doi{10.1103/PhysRevE.66.056108}.
\newblock \urlprefix\url{https://link.aps.org/doi/10.1103/PhysRevE.66.056108}

\bibitem{MR3671992}
Laskin, N.: Time fractional quantum mechanics.
\newblock Chaos Solitons Fractals \textbf{102}, 16--28 (2017).
\newblock \doi{10.1016/j.chaos.2017.04.010}.
\newblock \urlprefix\url{https://doi.org/10.1016/j.chaos.2017.04.010}

\bibitem{MR804935}
Lubich, C.: Fractional linear multistep methods for {A}bel-{V}olterra integral
  equations of the second kind.
\newblock Math. Comp. \textbf{45}(172), 463--469 (1985).
\newblock \doi{10.2307/2008136}.
\newblock \urlprefix\url{https://doi.org/10.2307/2008136}

\bibitem{MR2944107}
Malinowska, A.B., Torres, D.F.M.: Towards a combined fractional mechanics and
  quantization.
\newblock Fract. Calc. Appl. Anal. \textbf{15}(3), 407--417 (2012).
\newblock \doi{10.2478/s13540-012-0029-9}.
\newblock \urlprefix\url{https://doi.org/10.2478/s13540-012-0029-9}

\bibitem{MR3995304}
Massei, S., Mazza, M., Robol, L.: Fast solvers for two-dimensional fractional
  diffusion equations using rank structured matrices.
\newblock SIAM J. Sci. Comput. \textbf{41}(4), A2627--A2656 (2019).
\newblock \doi{10.1137/18M1180803}.
\newblock \urlprefix\url{https://doi.org/10.1137/18M1180803}

\bibitem{MR4235307}
Massei, S., Robol, L.: Rational {K}rylov for {S}tieltjes matrix functions:
  convergence and pole selection.
\newblock BIT \textbf{61}(1), 237--273 (2021).
\newblock \doi{10.1007/s10543-020-00826-z}.
\newblock \urlprefix\url{https://doi.org/10.1007/s10543-020-00826-z}

\bibitem{10.1063/1.1769611}
Naber, M.: Time fractional {S}chrödinger equation.
\newblock Journal of Mathematical Physics \textbf{45}(8), 3339--3352 (2004).
\newblock \doi{10.1063/1.1769611}.
\newblock \urlprefix\url{https://doi.org/10.1063/1.1769611}

\bibitem{Podlubny1999}
Podlubny, I.: {Fractional differential equations: an introduction to fractional
  derivatives, fractional differential equations, to methods of their solution
  and some of their applications}.
\newblock Mathematics in science and engineering. Academic Press, London (1999)

\bibitem{Poz24}
Pozza, S.: A new closed-form expression for the solution of {ODE}s in a ring of
  distributions and its connection with the matrix algebra.
\newblock Linear and Multilinear Algebra \textbf{0}(0), 1--11 (2024).
\newblock \doi{10.1080/03081087.2024.2303058}.
\newblock \urlprefix\url{https://doi.org/10.1080/03081087.2024.2303058}

\bibitem{Pozza2023a}
Pozza, S., Van~Buggenhout, N.: A new matrix equation expression for the
  solution of non-autonomous linear systems of {ODE}s.
\newblock PAMM \textbf{22}, e202200117 (2023).
\newblock \doi{10.1002/PAMM.202200117}.
\newblock
  \urlprefix\url{https://onlinelibrary.wiley.com/doi/full/10.1002/pamm.202200117
  https://onlinelibrary.wiley.com/doi/abs/10.1002/pamm.202200117
  https://onlinelibrary.wiley.com/doi/10.1002/pamm.202200117}

\bibitem{pozza2023new}
Pozza, S., Van~Buggenhout, N.: {A new Legendre polynomial-based approach for
  non-autonomous linear ODEs}.
\newblock ETNA - Electronic Transactions on Numerical Analysis pp. 292--326
  (2024).
\newblock \urlprefix\url{https://epub.oeaw.ac.at/?arp=0x003f234e}

\bibitem{ryckebusch2023frechetlie}
Ryckebusch, M., Bouhamidi, A., Giscard, P.L.: {A Fréchet Lie group on
  distributions}.
\newblock J. Math. Anal. Appl. \textbf{546}(1), 129195 (2025).
\newblock \doi{https://doi.org/10.1016/j.jmaa.2024.129195}.
\newblock
  \urlprefix\url{https://www.sciencedirect.com/science/article/pii/S0022247X2401117X}

\bibitem{saad03}
Saad, Y.: Iterative Methods for Sparse Linear Systems, 2nd edn.
\newblock Society for Industrial and Applied Mathematics, Philadelphia, PA
  (2003)

\bibitem{Sim16}
Simoncini, V.: Computational methods for linear matrix equations.
\newblock SIAM Rev. \textbf{58}, 377--441 (2016).
\newblock \doi{10.1137/130912839}.
\newblock \urlprefix\url{https://epubs.siam.org/doi/10.1137/130912839}

\bibitem{volterralecons}
Volterra, V., P\'er\`es, J.: Le\c{c}ons sur la composition et les fonctions
  permutables.
\newblock \'Editions Jacques Gabay, Paris (1928)

\bibitem{MR3671993}
Zhang, Y., Sun, H., Stowell, H.H., Zayernouri, M., Hansen, S.E.: A review of
  applications of fractional calculus in {E}arth system dynamics.
\newblock Chaos Solitons Fractals \textbf{102}, 29--46 (2017).
\newblock \doi{10.1016/j.chaos.2017.03.051}.
\newblock \urlprefix\url{https://doi.org/10.1016/j.chaos.2017.03.051}

\end{thebibliography}

\bigskip  %

\small %
\noindent
{\bf Publisher's Note}
Springer Nature remains neutral with regard to jurisdictional claims in published maps and institutional affiliations.

\end{document}